\documentclass{article}

\usepackage[
  paperheight=11in,
  paperwidth=8.5in,
  top=10mm, bottom=20mm, left=28mm, right=28mm
]{geometry}

\usepackage{amsmath,amsthm,amsfonts,amssymb,mathrsfs,dsfont, faktor}
\usepackage{mathtools} 
\usepackage{graphicx}
\usepackage{longtable}
\usepackage{float}
\usepackage{subcaption}
\usepackage{booktabs}
\usepackage{siunitx}  
\usepackage{csvsimple}  
\usepackage{epstopdf} 
\graphicspath{{graphics/}}

\usepackage{tikz}
\usetikzlibrary{arrows,trees,positioning}
\usepackage{quiver} 

\usepackage{enumerate}
\usepackage{comment}
\usepackage[colorinlistoftodos]{todonotes}

\usepackage[table, dvipsnames]{xcolor}
\usepackage{hyperref}
\hypersetup{
  colorlinks=true,
  citecolor=teal,
  linkcolor=MidnightBlue,
  urlcolor=black
}
\usepackage{cleveref} 

\usepackage[toc,page]{appendix}

\numberwithin{equation}{section}

\theoremstyle{plain}  
\newtheorem{theorem}{Theorem}[section]
\newtheorem*{theorem*}{Theorem}   
\newtheorem{corollary}{Corollary}[section]
\newtheorem{lemma}{Lemma}[section]
\newtheorem{proposition}{Proposition}[section]

\theoremstyle{definition}
\newtheorem{definition}{Definition}[section]
\newtheorem{remark}{Remark}[section]

\usepackage[
  backend=biber,
  style=alphabetic,
  doi=true,
  url=true,
  eprint=false,
  isbn=false
]{biblatex}

\renewbibmacro*{doi+eprint+url}{%
  \printfield{doi}%
  \newunit\newblock
  \iffieldundef{doi}{\printfield{url}}{}%
}

\newcommand{\N}{\mathbb{N}}

\newcommand{\Z}{\mathbb{Z}}

\newcommand{\Q}{\mathbb{Q}}
\newcommand{\F}{\mathbb{F}}

\newcommand{\ot}{\mathcal{O}}

\newcommand{\paren}[1]{\left( #1 \right)}
\newcommand{\kron}[2]{\left( \tfrac{#1}{#2} \right)}

\newcommand{\Redei}{R\'edei }
\newcommand{\e}{\epsilon}

\DeclareMathOperator{\Gal}{Gal}

\DeclareMathOperator{\rank}{rank}

\DeclareMathOperator{\Cl}{Cl} 

\DeclareMathOperator{\rk}{rk}
\DeclareMathOperator{\Am}{Am}
\newcommand{\Amst}{\Am_{st}}

\DeclareFontFamily{U}{wncy}{}
\DeclareFontShape{U}{wncy}{m}{n}{<->wncyr10}{}
\DeclareSymbolFont{mcy}{U}{wncy}{m}{n}
\DeclareMathSymbol{\Sh}{\mathord}{mcy}{"58}

\newcommand{\mfp}{\mathfrak p}

\newcommand{\mfc}{\mathfrak c}

\newcommand{\fk}{\mathfrak }

\title{On the Iwasawa $\lambda$-invariant of the cyclotomic $\Z_2$-extension of a family of real quadratic fields in which $2$ splits  }
\date{}
\author{Josu\'e \'Avila, Foivos Chnaras}

\begin{document}

\maketitle 

\begin{abstract}
We study Greenberg's conjecture for cyclotomic $\Z_2$-extensions of
real quadratic fields. Let $K=\mathbb Q(\sqrt{pq})$, where
$$
p\equiv 1 \mod 8,\qquad q\equiv 9 \mod {16},\qquad
\left(\frac{p}{q}\right)=-1.
$$
Under the additional assumptions
$$
\left(\frac{2}{p}\right)_4
\left(\frac{2}{q}\right)_4
\left(\frac{pq}{2}\right)_4=-1
$$
and
$$
\left(\frac{2}{p}\right)_4=-1
\quad\text{or}\quad
\left(\frac{2}{q}\right)_4=-1,
$$
we prove that $\lambda_2(K)=0$. The proof combines Greenberg's criterion
for the split prime case with a capitulation argument modeled on Kumakawa.
The main new input is a square-class computation of the Hasse unit index of
the biquadratic extension $K_2=\Q(\sqrt{pq}, \sqrt{2+\sqrt{2}})/\mathbf Q_1=\Q(\sqrt{2})$, showing that $q(K_2)\le 2$. 
\end{abstract}

\section{Introduction}
For a number field $K$ and a prime number $p$, we denote by $\mu_p(K), \lambda_p(K), \nu_p(K)$ the Iwasawa invariants of the cyclotomic $\Z_p$-extension $K_\infty$ of $K$. A standard result by Ferrero-Washington \cite{MR528968} states that $\mu_p(K)=0$ for any abelian number field $K$. Greenberg \cite{MR401702} conjectured that $\lambda_p(K)$ is also zero when $K$ is a totally real number field. Most of the progress on this conjecture has focused on the case of real quadratic number fields and specifically for $p=2$.

A classical theorem of Iwasawa states that if $K/\Q$ is a real quadratic extension and $p\nmid h_K$, then $\lambda_p(K)=0$. Whenever Iwasawa's result applies, we call this a "trivial case".

In this paper, we focus on the cyclotomic $\Z_2$-extension of real quadratic number fields. 
Genus theory and the theorem of \Redei- Reichardt imply that for $K=\Q(\sqrt{D})$, the following cases are "trivial" (in the sense that $2$ does not split in $K$ and the class number is odd): 

$$D=\begin{cases}
    2,\\ 
    p, \quad p \equiv 5 \mod 8, \\ 
    q, \quad q\equiv 3 \mod 4, \\ 
    pq, \quad p\equiv 3 \mod 8, \quad q\equiv 7 \mod 8
\end{cases}$$

A lot of research has been devoted to the study of the "non-trivial" cases, with various partial results. In this paper, we focus on the following case:
\begin{theorem} \label{th: main th}
Let $K=\mathbf{Q}(\sqrt{pq})$, where
$$
p\equiv 1 \mod{8},\qquad q\equiv 9 \mod{16},\qquad \left(\frac{p}{q}\right)=-1.
$$
Assume
$$
\left(\frac{2}{p}\right)_4\left(\frac{2}{q}\right)_4\left(\frac{pq}{2}\right)_4=-1
$$
and
$$
\left(\frac{2}{p}\right)_4=-1
\quad\text{or}\quad
\left(\frac{2}{q}\right)_4=-1.
$$
Then $\lambda_2(K)=0$.
\end{theorem}

The proof has two main parts. First, using a capitulation argument modeled on
Kumakawa \cite{MR4262274}, we show that it is enough to prove $q(K_2)<4$, where $q(K_2)$
is the Hasse unit index of the biquadratic extension $K_2/\mathbf Q_1$, with $K_2$ being the second layer of $K_{\infty}$ and $\mathbf{Q}_1=\Q (\sqrt{2})$. Second, let $F_1=\Q (\sqrt{(2+\sqrt{2})D}$ and $\mathbf{Q}_2=\Q(\sqrt{2+\sqrt{2}})$. We
compute the Hasse index by a square-class analysis of
$$
E(K_1)E(F_1)E(\mathbf{Q}_2)\subseteq E(K_2),$$ 

where $E(L)$ denotes the units of the ring of integers $\ot_L$, and prove that $q(K_2)\le 2$.

\medskip

\noindent\textbf{Organization of the paper.}
In Section~3, we recall known cases and place Theorem~1.1 among the
remaining cases for real quadratic fields. Section~4 reviews Greenberg's
criterion and the ambiguous class groups used later. Section~5 collects the
genus theory, R\'edei matrix, Hilbert symbol, and biquadratic class number
formulae needed in the proof. Section~6 treats an auxiliary case with
$(p,q)\equiv (5,1)\mod 8$ (See Proposition $\ref{th: case (5,1)})$. Finally, Section~7 proves Theorem~1.1: we first
analyze the fields $K_1$ and $F_1$, then use a Kumakawa-type
capitulation argument to reduce the proof to bounding the Hasse unit index
$q(K_2)$, and conclude by proving $q(K_2)\le 2$.

\section{Notation}

We establish the following notation that is used throughout this article:
\begin{enumerate}
    \item  $K=\Q (\sqrt{D})$, where the conditions on $D$ will be specified when needed. 
    \item $\epsilon_D$ is the fundamental unit of $K=\Q(\sqrt{D})$ and  $\epsilon_2=1+\sqrt{2}$ is the fundamental unit of $\Q(\sqrt{2})$.
    \item $K_n$ is the $n-\rm{th}$ layer of the cyclotomic $\Z_2$-extension of $K$ and $\mathbf{Q}_n$ is the $n-\rm{th}$ layer of the cyclotomic $\Z_2$-extension of $\Q$. In particular, $K_0 = K, K_1 = \Q(\sqrt{2}, \sqrt{D})$ and $K_{2}=\Q \left( \sqrt{2+\sqrt{2}},\sqrt{D}\right)$, while $\mathbf{Q}_{1} = \Q(\sqrt{2})$ and $\mathbf{Q}_2=\Q(\sqrt{2+\sqrt{2}})$. In general, $K_n=\Q(a_n, \sqrt{D})$ where $a_0=0$ and $a_n=\sqrt{2+a_{n-1}}$.
    \item $F_n$ is the subfield of $K_{n+1}$ containing $\mathbf{Q}_n$, different from $K_n$ and $\mathbf{Q}_{n+1}$. In particular, $F_0= \Q( \sqrt{2D})$ and $F_1=\Q(\sqrt{(2+\sqrt{2})D})$.
    \item $E(L)$ is the group of units of the ring of integers $\ot_L$ of a number field $L$.
    \item $A(L)$ is the $2$-primary part of the class group of $L$. We denote $A_n=A(K_n)$.
\end{enumerate}

Any other notation needed is specified in the respective section.

\section{Overview of the cases}

\subsection{Previous results}
In this section we record various past results. 
\\

As already mentioned in the introduction, when $K=\Q(\sqrt{D})$, the following are "trivial cases":  
$$D=\begin{cases}
    2,\\ 
    p, \quad p \equiv 5 \mod 8, \\ 
    q, \quad q\equiv 3 \mod 4, \\ 
    pq, \quad (p,q)\equiv (3,7) \mod 8
\end{cases}$$
We also remark that $K=\Q(\sqrt{D})$ and $F_0=\Q(\sqrt{2D})$ share the same $\Z_2$-extension and therefore $\lambda_2(K)=0 \iff \lambda_2(F_0)=0$.

\begin{theorem}
    In the following non-trivial cases, we have $\lambda_2(K)=0$ for $K=\Q(\sqrt{D})$:

    \begin{enumerate}[(1)]
        \item $D=p, \quad p \equiv 1\mod 8$ and $\paren{ \frac{2}{p}}_4 \paren{\frac{p}{2}}_4 =-1$. Then also $\nu_2(K)=0$.
        \item $D=pq, \quad (p,q)\equiv (3,3) \mod 8$. Then also $\nu_2(K)=0$.
        \item $D=pq, \quad (p,q)\equiv (3,5) \mod 8$. Then also $\nu_2(K)>0$.
        \item $D=pq, \quad (p,q)\equiv (5,7) \mod 8$. Then also $\nu_2(K)>0$.
        \item $D=pq, \quad (p,q)\equiv (5,5) \mod 8$. Then also $\nu_2(K)>0$.
        \item $D=pq, \quad (p,q)\equiv (3,1) \mod 8$ and $\paren{\frac{q}{p}}=-1$ and $\kron{2}{q}_4 =-1$
        \item $D=pq, \quad (p,q)\equiv (3,1) \mod 8$, $\paren{\frac{q}{p}}=-1$, $q\equiv 9 \mod 16$, $\paren{\frac{2}{q}}_4=1$ and $\rk_4(A_2)=1$. 
    \end{enumerate}
\end{theorem}
Cases (1)-(5) are due to Ozaki-Taya \cite{MR1484637} and case (6) is due to Fukuda-Komatsu \cite{MR2149635}. Case (7) is due to \cite{MR4262274}.\\

Consider now $K=\Q(\sqrt{pq})$ with $p\equiv 5\mod 8$ and $q\equiv 1 \mod 8$. Then the following cases have been proven: \cite[Cor~ 3.3.3]{MR4879114}

\begin{corollary}
\begin{enumerate}
    \item Assume $\paren{\frac{p}{q}}=-1$. Then $\lambda_2(K)=0$ and $A_n=A_0\simeq \Z/2\Z$ for all $n\geq 0$.
    \item Assume $\paren{\frac{p}{q}}=1$ and $\paren{\frac{p}{q}}_4=\paren{\frac{q}{p}}_4=-1.$ Then $\lambda_2(K)=0$ and $A_n\simeq A_0 \simeq \Z / 4 \Z$ for all $n\geq 0$
    \item Assume $\paren{\frac{p}{q}}=1$ and $\paren{\frac{p}{q}}_4=\paren{\frac{q}{p}}_4=1$ and $N(\epsilon_{D})=-1$. Then $\lambda_2(K)=0$ and $A_n \simeq A_0$ for all $n\geq 0$.
    \item Assume $\paren{\frac{p}{q}}=1$ and $\paren{\frac{p}{q}}_4=-\paren{\frac{q}{p}}_4$ and $q\equiv 1 \mod 16$ and $N(\epsilon_{2D})=-1$. Then $\lambda_2(K)=0$ and $A_n\simeq A_1$ for all $n\geq 1$.
     
\end{enumerate}
\end{corollary}

\subsection{Remaining Cases}
When $K= \Q(\sqrt{pq})$, there are only a few remaining cases for proving that $\lambda_2(K)=0$. These are:

\begin{enumerate}
    \item $(p,q)\equiv(3,1) \mod 8$ and $\paren{\frac{q}{p}}=1$ or $2^{\frac{q-1}{4}}\equiv 1 \mod q$ 
    \item $(p,q) \equiv (7,7) \mod 8$
    \item $(p,q)\equiv(1,1) \mod{8}$.
    \item $(p,q) \equiv (5,1) \mod 8$ and $\paren{\frac{2}{q}}_4= \paren{\frac{q}{2}}_4$.
    \item $(p,q) \equiv (5,1) \mod 8$ and $\paren{\frac pq}=1$, $\paren{\frac{2}{q}}_4 \neq \paren{\frac{q}{2}}_4$, $\paren{\frac{p}{q}}_4= \paren{\frac{q}{p}}_4=1$ and    $N(\epsilon_D)=1$
    \item $p \equiv 5 \mod 8, \quad q \equiv 1 \mod 16$, $\paren{\frac{2}{q}}_4 \neq \paren{\frac{q}{2}}_4$, $\paren{\frac{p}{q}}_4= -\paren{\frac{q}{p}}_4$  and $N(\epsilon_{2D})=1$
    \item $p \equiv 5 \mod 8, \quad q \equiv 9 \mod 16$, $\paren{\frac{2}{q}}_4 \neq \paren{\frac{q}{2}}_4$,  $\paren{\frac{p}{q}}_4= -\paren{\frac{q}{p}}_4$.
    
\end{enumerate}

\section{Greenberg's criterion}
In this section, we recall Greenberg's criterion for the vanishing of the $\lambda$- invariant. 
\subsection{Strongly Ambiguous Group}
We start by recalling some basic definitions and facts.
\begin{definition}
Let $L/K$ be a Galois extension of number fields . If $\mfc$ is an ideal in $L$, we denote by $[\mfc]$ its class in the class group $Cl(L)$.
    \begin{enumerate}
        \item Define $Am(L/K)=\{[\mfc] \in \Cl(L): [\mfc]^\sigma = [\mfc]$ for all $\sigma \in \Gal(L/K)$ \}.
        \item Define $\Amst(L/K) = \{ [c] \in \Cl(L): \mfc^\sigma=\mfc$ for all $\sigma \in \Gal(L/K) \}$.
    \end{enumerate}
\end{definition}
When we're working with the cyclotomic $\Z_p$-extension $K_\infty/K$, we denote $\Am(K_n/K)_p, \Amst(K_n/K)_p$ simply by $B_n$ and $B_n'$ respectively, where $\Am(K_n/K)_p, \Amst(K_n/K)_p$ represent the $p$-parts.\\

An important theorem of Chevalley states the following: 

\begin{theorem}{\textbf{Chevalley's Theorem}} \label{th: Chevalley}
Suppose $L/K$ is a cyclic Galois extension of number fields of degree $n$. Then \begin{enumerate}[(1)]
\item $$\Amst(L/K)=h_K \frac{\prod\limits_\fk p e_\fk p}{n [E(K):N(E(L))] }$$
\item $$\Am(L/K)=h_K \frac{\prod\limits_\fk p e_\fk p}{n [E(K):E(K)\cap N(L^\times)] }$$
\end{enumerate}
    where the product varies over all places (archimidean and non-archimidean) of $K$, $e_\fk p$ represents the ramification index, $h_K$ is the class number of $K$ and $N$ is the norm map.
\end{theorem}
\begin{proof}
    See \cite{lang2012cyclotomic}
\end{proof}

\subsection{Greenberg's criterion}

\begin{theorem}[Greenberg] \label{th: Greenberg split}
    Fix a prime $p$ and assume $K$ is a real quadratic number field such that there exists a unique prime $\mfp$ above p and that it's totally ramified in the cyclotomic $\Z_p$- extension $K_\infty/K$. Then 
    $$ B_n'=1 \text{ for some } n \in \N \Rightarrow \mu_p(K)=\lambda_p(K)=0.$$ 
    Also, if $A_0$ capitulates in $K_\infty$, then the same conclusion holds. 
\end{theorem}

\begin{proof}
    This follows from \cite[Theorem~1]{MR401702}. We briefly sketch the argument here.

    Let $i_{0,n}: A_0\to A_n$. Write $\mfp$ for the prime in $K$ above $p$ and $D_n$ for the subgroup of $A_n$ consisting of ideal classes which divide  $p$. 
    
    \textbf{Claim:} $B_n'= i_{0,n}(A_0) D_n$. 
    
    Assuming the claim, $B_n'=1$ implies that $i_{0,n}(A_0)=1$. Therefore, every ideal class in $A_0$ becomes eventually principal. Hence, \cite[Theorem~1]{MR401702} implies that $\mu_p(K)=\lambda_p(K)=0$.

    \textbf{Proof of claim:} Since $\fk p$ is totally ramified in $K_\infty/K$, the primes of
$K_n$ above $\fk p$ are fixed by $\Gal(K_n/K)$. Hence,
$D_n\subseteq B_n'$. Also, the extension of any ideal class from
$K$ to $K_n$ is strongly ambiguous, so
$$
i_{0,n}(A_0)D_n\subseteq B_n'.
$$

Conversely, let $[\fk c]\in B_n'$, and choose a representative
$\fk c$ fixed by $\Gal(K_n/K)$. Split
$$
\fk c=\fk c_{\fk p}\fk c',
$$
where $\fk c_{\fk p}$ is supported at the primes above
$\fk p$, and $\fk c'$ is prime to $\fk p$. Then
$[\fk c_{\fk p}]\in D_n$. Since $K_n/K$ is unramified outside
$\fk p$, the $\Gal(K_n/K)$-invariant ideal $\fk c'$ is the extension of
an ideal $\fk b'$ of $K$, say
$$
\fk c'=\fk b'\mathcal O_{K_n}.
$$
Thus
$$
[\fk b'\mathcal O_n]
=
[\fk c]\,[\fk c_{\fk p}]^{-1}\in A_n.
$$
Taking norms to $K$, we obtain
$$
N_{K_n/K}\bigl([\fk b'\mathcal O_n]\bigr)
=
[\fk b']^{[K_n:K]}
=
[\fk b']^{p^n}.
$$
Since the left-hand side is the norm of a class in $A_n$, it lies in
$A_0$. Hence
$$
[\fk b']^{p^n}\in A_0.
$$
But if a class $x\in \Cl(K)$ satisfies $x^{p^n}\in A_0$, then its
prime-to-$p$ part must be trivial; therefore $x\in A_0$. Applying this to
$x=[\fk b']$, we get $[\fk b']\in A_0$. Consequently
$$
[\fk c']
=
[\fk b'\mathcal O_n]
\in i_{0,n}(A_0),
$$
and since $[\fk c_{\fk p}]\in D_n$, we conclude that
$$
[\fk c]\in i_{0,n}(A_0)D_n.
$$
Therefore,
$$
B_n'=i_{0,n}(A_0)D_n.
$$
    
\end{proof}

 The splitting case is more subtle. Recall the decomposition we saw in the above proof $$B_n'=i_{0,n}(A_0)D_n.$$

Greenberg proves the following

\begin{theorem}[Greenberg] \label{th: Greenberg non split}
    Let $K$ be a totally real number field  in which $p$ splits completely. Assume Leopoldt's conjecture is valid for $K$. Then $$\lambda_p(K)=\mu_p(K)=0 \iff B_n=D_n \quad \text{for all n sufficiently large}  $$
\end{theorem}

As before, $E(K)$ denotes the group of units of $K$ and $U_v$ the group of local units of $K_v$, for some valuation $v$. Let $$\rho_n:E(K)/E(K)^{p^n} \to \prod_{v|p} U_v/U_v^{p^n}$$ be the natural map induced by the embeddings $K \hookrightarrow K_v$.
\\
We define the following property: 
\begin{equation}
    \label{eq: property Pn}
    (P_n): \quad \ker(\rho_n)=0
\end{equation}
Finally, on the cycolotomic $\Z_p$-extension $K_\infty/K$, we define  $$H_n=\ker(A_n\to A_\infty)$$\\
Greenberg proves:

\begin{theorem}[Greenberg]\label{th: greenberg split Bn}
Let $K$ be a totally real number field. Assume $p$ splits completely in $K$ and that $K$ satisfies the property $(P_n)$ for all $n$. Then $$B_n=B_n'$$ If furthermore $H_0=A_0$, then $B_n=D_n$ for all n suffiently large and $\lambda_p(K)=\mu_p(K)=0$
    
\end{theorem}

\begin{proof}
    See \cite[Theorem~2]{MR401702} and the comments following it.
\end{proof}

\begin{remark}\label{rmk: property P_n}
    Our cases of interest are real quadratic field $K=\Q(\sqrt D)$ and $p=2$. In this case, it is easy to check that the property $(P_n)$ holds for all $n$ if and only if the fundamental unit $\epsilon_D$ satisfies that $$v_2(\log_p(\epsilon_D))=2$$. This is proven in Lemma \ref{lem: property P_n}
\end{remark}

\section{Genus theory and \Redei Matrix}
In this section, we recall the definition and some properties of the \Redei matrix associated to a real quadratic number field. We start with some important results on biquadratic extensions. 

\subsection{Biquadratic Extensions}

We define a biquadratic extension as any extension of number fields $L/K$ such that
$$
\mathrm{Gal}(L/K) \simeq \Z / 2 \Z \times \Z / 2 \Z.
$$
Most of the work on imaginary biquadratic extensions $L/\mathbb{Q}$ was done by
Kuroda \cite{MR21031}. For real biquadratic extensions, we have the following result
by Kubota, commonly known as the Kubota--Kuroda formula.

\begin{theorem}[Kubota--Kuroda {\cite{MR83009}}] \label{kuroda}
Let $L/\mathbb{Q}$ be a totally real biquadratic extension, with unit group $E(L)$
and $2$-class group $A(L)$. Let $L_1, L_2,$ and $L_3$ be the quadratic subfields of $L$.
Let $\epsilon_i$ be the fundamental unit of $L_i$, for $i=1,2,3$.
Let
$$
q(L) := \bigl[E(L) : \langle -1, \epsilon_1, \epsilon_2, \epsilon_3 \rangle \bigr]
$$
be the Hasse unit index of $L$. Then
$$
|A(L)| = \frac{1}{4}\, q(L)\, |A(L_1)|\, |A(L_2)|\, |A(L_3)|.
$$
\end{theorem}

Further, a system $\Gamma$ of fundamental units of $L$ (F.S.U.) is one of the following
possibilities:
\begin{enumerate}
  \item $\{\epsilon_1, \epsilon_2, \epsilon_3\}$,
  \item $\{\sqrt{\epsilon_1}, \epsilon_2, \epsilon_3\}$,
  \item $\{\epsilon_1, \sqrt{\epsilon_2}, \epsilon_3\}$,
  \item $\{\sqrt{\epsilon_1 \epsilon_2}, \epsilon_2, \epsilon_3\}$,
  \item $\{\sqrt{\epsilon_1 \epsilon_2}, \epsilon_2, \sqrt{\epsilon_3}\}$,
  \item $\{\sqrt{\epsilon_1 \epsilon_2}, \sqrt{\epsilon_1 \epsilon_3}, \sqrt{\epsilon_2 \epsilon_3}\}$,
  \item $\{\sqrt{\epsilon_1 \epsilon_2 \epsilon_3}, \epsilon_2, \epsilon_3\}$.
\end{enumerate}

Any unit $\epsilon_i$ that appears under the square root sign is assumed to have
norm equal to $1$, except for the last case, where all the units can have the same
norm, either all $1$ or all $-1$.\\

The theorem has been generalized to biquadratic extensions over arbitrary real number fields by Lemmermeyer \cite{MR1276992}. We get the following for a biquadratic extension over $\mathbf {Q}_1$:
\begin{theorem} \label{th: lemmermeyer biquadratic class number}
    Let $L/\mathbf Q_1$ be a biquadratic extension. Let $L_i$ for $i=1,2,3$ be the three quadratic subextensions of $L/\mathbf Q_1$ and $h_i=h(L_i)$. Then $$h(L)=\frac 18 q(L) \cdot h_1\cdot h_2 \cdot h_3,$$ where $q(L)=[E(L):E(L_1)E(L_2)E(L_3)]$ is the Hasse unit index.\\
    
    Furthermore, if $L=F(\sqrt{d})$ and $u \in F^\times$, then $$u \in L^2 \iff \begin{cases} u \in F^2 \text{ or } \\ 
        u \in dF^2 
    \end{cases}$$ 
\end{theorem}

\subsection{Genus Theory}
 We recall some standard results in classical genus theory:

\begin{theorem}[Genus Theory]\label{th: Genus theory}
Let $K=\Q(\sqrt{D})$ be a real quadratic number field with discriminant $d_K$. Then  
$$rk_{2}(Cl^+(K))=\omega (d_{K})-1,$$
where $\omega (d_{K})=\#\{p | d_{K}\}$ and $Cl^+(K)$ denotes the narrow class group. Also,
$$[Cl^{+}(K):Cl(K)]=\begin{cases}
     1 \quad N(\epsilon_{D})=-1 
     \\ 2 \quad N(\epsilon_{D})=1 
    \end{cases}$$

\end{theorem}

Naturally, we need a criterion to determine the value of the norm of the fundamental unit. Scholz comes to the rescue:

\begin{theorem}[Scholz, \cite{MR1545490, MR309898}]\label{th: Nm(ed)=-1}
Let $D=pq$ with $p,q$ odd primes and $p\equiv q\equiv 1 \mod 4$.
Put $K=\Q(\sqrt{D})$, let $N=N_{K/\Q}$ be the norm,
let $h$ (resp.\ $h^+$) be the (resp.\ narrow) class number of $K$,
and let $\epsilon_D$ be a fundamental unit of $K$.

\begin{enumerate}
\item If $\left(\frac{p}{q}\right)=-1$, then
$h \equiv h^+ \equiv 2 \mod 4,
\qquad\text{and}\qquad
N(\epsilon_D)=-1.$

\item If $\left(\frac{p}{q}\right)=1$, then:
\begin{enumerate}
\item[(i)] If $\left(\frac{p}{q}\right)_4 = -\left(\frac{q}{p}\right)_4$, then $h^+ \equiv 4 \mod 8,
\qquad 2h \equiv 4 \mod 8,
\qquad\text{and}\qquad
N(\epsilon_D)=1.$

\item[(ii)] If $\left(\frac{p}{q}\right)_4=\left(\frac{q}{p}\right)_4=-1$, then
$h \equiv h^+ \equiv 4 \mod 8,
\qquad\text{and}\qquad
N(\epsilon_D)=-1.$

\item[(iii)] If $\left(\frac{p}{q}\right)_4=\left(\frac{q}{p}\right)_4=1$, then
$h^+ \equiv 0 \mod 8.$
\end{enumerate}
\end{enumerate}

\end{theorem}

We therefore have a complete classification of the $2$-rank of the class group. For the $4$-rank, we use the \Redei matrix.

\subsection{\Redei-Reichardt Matrix, \texorpdfstring{\cite{MR1581397}}{cite}}
Let $K=\Q(\sqrt{D})$ with discriminant $d_K= \prod_{i=1}^r \ell_i$. Define $$
\ell^*:=(-1)^{(\ell-1)/2}\,\ell = \begin{cases}
\phantom{-}\ell,& \ell\equiv 1\mod 4 \\
-\ell & \ell\equiv 3\mod 4.
\end{cases}
$$
and if $2\mid d_K$, define
$$
2^*=\begin{cases}
-4 & D\equiv 3\mod 4 \\
\phantom{-}8 & D\equiv 2\mod 8 \\
-8 & D\equiv 6\mod 8.
\end{cases}
$$

Define the $r\times r$ matrix $R=(a_{ij})$ over $\F_2$ by:,
$$
(-1)^{a_{ij}}=\Bigl(\frac{\ell_j^*}{\ell_i}\Bigr) \quad \text{for} \quad i \neq j,
$$ 
where $\bigl(\frac{\cdot}{\ell_i}\bigr)$ is the Kronecker symbol,
and the diagonal entries are chosen so that each row sums to $0$ in $\F_2$:
$$
a_{ii}=\sum_{j\neq i} a_{ij}.
$$
Then 
$$
r_4\!\bigl(\Cl^+(K)\bigr)=r-1-\rank_{\F_2}(R)
= r_2\!\bigl(\Cl^+(K)\bigr)-\rank_{\F_2}(R).
$$

Specializing this to the case of interest: \\

Let $D=pq$ be squarefree and odd.
\begin{itemize}
\item If $pq\equiv 1\mod 4$, then $d_K=pq$ and 
$$
R=\begin{pmatrix} a & a \\ a & a \end{pmatrix},
\qquad
\rank(R)=
\begin{cases}
0,& \bigl(\frac{q^*}{p}\bigr)=+1,\\
1,& \bigl(\frac{q^*}{p}\bigr)=-1,
\end{cases}
$$
and therefore
\begin{equation} \label{eq: rk4 for 1 mod 4}
r_4\!\bigl(\Cl^+(K)\bigr)=1-\rank(R)=
\begin{cases}
1,& \bigl(\frac{q^*}{p}\bigr)=+1,\\
0,& \bigl(\frac{q^*}{p}\bigr)=-1.
\end{cases}
\end{equation}

\item If $pq\equiv 3\mod 4$, then $d_K=4pq$ and $r=3$. Assume, without loss of generality, that $(p,q)=(1,3) \mod 4$.
A simple case by case study shows that $$
\rank_{\F_2}(R)=
\begin{cases}
2, & \left(\dfrac{p}{q}\right)=-1,\\
1, & \left(\dfrac{p}{q}\right)=+1 \ \text{and}\ p\equiv 1\mod 8, \\
2, & \left(\dfrac{p}{q}\right)=+1 \ \text{and}\ p\equiv 5\mod 8.
\end{cases}
$$

Consequently,
\begin{equation} \label{eq: rk4 for 3 mod 4}
r_4 \bigl(\Cl^+(K)\bigr)= 2-\rank_{\F_2}(R)=
\begin{cases}
0, & \left(\dfrac{p}{q}\right)=-1,\\
1, & \left(\dfrac{p}{q}\right)=+1 \ \text{and}\ p\equiv 1\mod 8, \\
0, & \left(\dfrac{p}{q}\right)=+1 \ \text{and}\ p\equiv 5\mod 8.
\end{cases}
\end{equation}

\end{itemize}

\subsection{Hilbert Symbols}
We plan to generalize the theory of \Redei matrices to any quadratic extension $E/F$ with $h_F=1$. For this, we need to recall Hilbert symbols briefly.

Let $K/\Q_p$ be a local field with place $v$ and with $p$ odd. For $r,s \in K^\times$, the Hilbert symbol $(r,s)_p$ is defined as follows: 
Write $r=\pi^a u$ and $s=\pi^b v$ where $\pi$ is a uniformizer of $K$ and $u,w$ are units. Then
\begin{equation}\label{eq: hilber symbol odd prime}
(r,s)_v = (-1)^{ab \frac{N-1}{2}} \kron{u}{\pi}^b \kron{v}{\pi}^a
\end{equation} where $N=|\ot_K/\pi|$.
\\
Over the $2$-adics, again writing $r= 2^a u$ and $s=2^b v$, we have 
\begin{equation}
    {\displaystyle (r,s)_{2}=(-1)^{{\frac{u-1}{2}\frac{v-1}{2}+a \frac{v-1}{2}+b \frac{u^2-1}{8}}} } 
\end{equation} 
Quadratic reciprocity gives us $$\prod_v (r,s)_v=1$$ where the product is taken over all places of $K$.

\subsection{Generalized \Redei matrix \texorpdfstring{\cite{MR2452635}}{citation}}

Let $F$ be a real number field with class number $h_F=1$ and $E/F$ a real quadratic extension, with $E=F(\sqrt{\delta})$. Write $\Gal(E/F)=\{1,\sigma\}$, $S=\{\fk p_1,\dots,\fk p_m\}$ for the set of all primes of $F$ which ramify in $E$ and $E(F)=<\epsilon_1,\dots,\epsilon_{r}>$ for the unit group of $F$. Organize the units such that $\epsilon_1,\cdots,\epsilon_\ell$ are in $E(F)\cap N(E^\times)$ and $\epsilon_{\ell+1},\cdots,\epsilon_r$ are non-norm units.

\textbf{Norm units:}
For a norm unit $\epsilon = N(x)$, we have by Hilbert's 90 that $x\ot_E =B^{\sigma-1}$ for some ideal $B$ of $E$. We can find $b\in F^\times$ such that $b B^{\sigma+1} = \ot_E$.  For each $\e_i$ that is a norm unit, set $b_i \in F^\times$ for that element.

\textbf{Primes $\fk p_i$}: For each ramified prime $\fk p_i=(\pi_i)$, set $a_i = \pi^{-1}$. 

The generalized \Redei matrix $R_{E/F}$ is defined as the $m\times r$ matrix over $F_2$ 

\begin{equation} \label{eq: generalized Redei}
    R_{E/F}=\begin{pmatrix}
        (a_1,\delta)_{\fk p_1} & \cdots & (b_1,\delta)_{\fk p_1} & \cdots & (\epsilon_{\ell+1},\delta)_{\fk p_1} & \cdots & (\epsilon_r,\delta)_{\fk p_1} \\
        & \cdots & & \cdots & & \cdots & \\
        (a_1,\delta)_{\fk p_m} & \cdots & (b_1,\delta)_{\fk p_m} & \cdots & (\epsilon_{\ell+1},\delta)_{\fk p_m} & \cdots & (\epsilon_r,\delta)_{\fk p_m}
    \end{pmatrix}
\end{equation}

\begin{theorem}[{\cite[Theorem~2.1]{MR2452635}}] \label{th: generalized Redei}
    Assume $h_F=1$. Then $r_4(\Cl^+(E))=m-1- \rank_{\F_2}(R_{E/F})$.
\end{theorem}

\begin{remark}
    As a convention, we consider the elements of the matrix $R_{E/F}$ as elements of $\F_2$ by identifying $0$ with $1$ and $1$ with $-1$.
\end{remark}

The $2$- rank can also be computed in a similar manner. We form the $m\times r$ matrix 
$M_{E/F}=((\e_j,\delta)_{\fk p_i})$. 
\begin{lemma}[{\cite[Lemma~2.4]{MR2452635}}]
With the above notation: 
$$r_2(\Cl^+(E))=m-1-\rank_{\F_2}(M_{E/F})$$
\end{lemma}

\section{Case: \texorpdfstring{$(p,q)\equiv (5,1)\mod 8$}{(p,q)=1 mod 8}}
Let $K=\Q(\sqrt{pq})$. In this section, we assume $p\equiv 5 \mod 8$ and $q\equiv 1 \mod 8$ with $\kron{q}{p}=1$. Similar to the techniques used in \cite{MR4879114}, we can prove the following:

\begin{proposition} \label{th: case (5,1)}
    Let $(p,q) \equiv (5,1)\mod 8$ and $K=\mathbb{Q}(\sqrt{D})$ with $D=pq$. Then $\lambda_2(K)=0$ for the following cases:
    \begin{enumerate}
        \item $\left( \dfrac{p}{q} \right)=-1$.   
\smallskip
\item  $\left( \dfrac{p}{q} \right)=1$, $\left( \dfrac{p}{q} \right)_{4}=-\left( \dfrac{q}{p} \right)_{4}$ and $N(\epsilon_{2D})=-1$.
    \end{enumerate}
\end{proposition}

\begin{proof}
    In both cases, there is only one prime above $2$, and we have $A_{0} \cong \Z/ 2\Z$, due to Theorem \ref{th: Nm(ed)=-1}. Frome Chevalley's formula, we have 
    $$| B_1'| =\frac{2}{[E(K):N (E(K_1))]}.$$

    From \cite{MR4879114}, when $ \left( \dfrac{p}{q} \right) =-1$ the fundamental system of units of $K_1$ is $\{ \epsilon_D, \epsilon_{2D}, \epsilon_2 \}$. For the second case, since $N(\epsilon_{2D})=-1$ and $N(\epsilon_D)=1$ , the only possible unit that is a square root is $\epsilon_D$. From \cite[Lemma~3.3.8]{MR4879114}, $\sqrt{\epsilon_D} \not \in K_1$ and $\{ \epsilon_D, \epsilon_{2D}, \epsilon_2 \}$ is a fundamental system of units of $K_1$. In any case, it follows that $\epsilon_D \not \in N(E(K_1))$ and the index is $2$, therefore $|B_1'|=1$ and we conclude from Theorem \ref{th: Greenberg split} that $\lambda_2(K) = 0$.
    
\end{proof}

This idea is pretty straightforward; we are exploiting the fact that the class group of the base field is $\Z / 2\Z$, and therefore $|B_1'|$ has only two possible values.

\section{Case: \texorpdfstring{$(p,q) \equiv (1,1) \mod 8$}{(p,q)= (1,1) mod 8}}
As always, let $K=\Q(\sqrt{D})$ where $D=pq$. In this section, we assume $p\equiv q \equiv 1 \mod 8$ with $\kron{p}{q}=-1$. Our goal in this chapter is to find a family that satisfies $\lambda_2(K)=0$. 

\begin{lemma} \label{lem: A0=C2}
    Assuming $\kron{p}{q}=-1$, we have $A_0=<\fk p>=<\fk q> \simeq \Z/2\Z$ where $\fk p$ and $\fk q$ are the primes in $K$ above $p$ and $q$ respectively.
\end{lemma}

\begin{proof}
    By Theorem~\ref{th: Nm(ed)=-1} we have that $N(\epsilon_D)=-1$. Therefore, by Theorem~\ref{th: Genus theory} we have that \\ $rk_2(\Cl^+(K))=1$ and $\Cl^+(K)=\Cl(K)$. By \ref{eq: rk4 for 1 mod 4}, we have that $rk_4(\Cl^+(K))=0$. Therefore, $A_0\simeq \Z/2\Z$ and the unique non-trivial class is generated by the prime above $p$ (or equivalently by the prime above $q$). Indeed, if $\fk p = (\pi)$ was principal, and $\pi = \frac{a+b \sqrt{d}}{2}$, then $N(\pi) = \frac{a^2-Db^2}{4}= \pm p$. Therefore, 
    $$a^2-Db^2 = \pm 4p $$
    This implies that $a^2 \equiv \pm 4 p \mod q$. But since $q\equiv 1 \mod 4$ : 

    $$\Big(\frac{\pm 4p}{q}\Big)=\Big(\frac{\pm p}{q}\Big)
=\Big(\frac{\pm 1}{q}\Big)\Big(\frac{p}{q}\Big)
=1\cdot (-1)=-1$$ which is a contradiction.
\end{proof}

Since $2$ splits in $K$, by Chevalley's formula (Theorem~\ref{th: Chevalley}), we have 

\begin{equation}
    |B_n|=\frac{2^{n+1}}{[E(K):E(K)\cap N(K_n^\times)]} \qquad |B_n'|=\frac{2^{n+1}}{[E(K):N(E(K_n))]}
\end{equation}

Numerical observations suggest that $\epsilon_D$ is never a norm from $E(K_n)$, although proving this is challenging. Instead, what we can show is the following:

\begin{lemma} \label{lem: ed is norm iff}
    Let $\epsilon_D$ be the fundamental unit of $K$. Then $$\epsilon_D \notin N(K_1^\times) \iff \kron{2}{p}_4 \kron{2}{q}_4 \kron{pq}{2}_4 =-1.$$
\end{lemma}

\begin{proof}
    The proof is reduced to local symbols computations. We have that 
    $$\epsilon_D \in N(K_1^\times) \iff (\epsilon_D,2)_\fk P =1 \quad \text{ for all } \fk P | 2.$$  Now, since $2$ splits in $K$, say $2\ot_K = \fk P_1 \fk P_2$, by the product formula, we have that 
    $$\epsilon_D \in N(K_1^\times) \iff (\epsilon_D,2)_{\fk P_1} = 1. $$

    We use the technique developed in \cite{MR4940532}: We define the field $F=\Q(\sqrt{\epsilon_D \sqrt{D}})$. This is a cyclic extension of $\Q$ of degree $4$, unramified outside $pq$, and $K$ is its quadratic subfield. We let $\gamma$ denote the generator of $\Gal(F/K)$. 

    By the product formula of Artin symbols, we have
    $$ \left(\frac{2,F/\Q}{p}\right)\left(\frac{2,F/\Q}{q}\right)\left(\frac{2,F/\Q}{2}\right)=1
\quad\text{in }\mathrm{Gal}(F/ \Q). $$ 
Moreover, since $2$ splits in $K$, the restriction of $\left(\frac{2,F/\mathbb Q}{v}\right)$ to $K$ is trivial for $v=p,q,2$. Hence these local symbols actually lie in $\Gal(F/K)=\{1,\gamma\}$. 

By \cite[Lemma~4]{MR4940532} : if $F$ is cyclic quartic and $p$ ramifies totally in $F$, then for $m$ prime to $p$, 
$$\left(\frac{m,F/\Q}{p}\right)=1
\iff \left(\frac mp\right)_4=1$$

In our case, we conclude that 

$\begin{array}{c} \left(\frac{2,F/\Q}{p}\right)=
\begin{cases}
1,&\left(\frac2p\right)_4=+1,\\
\gamma,&\left(\frac2p\right)_4=-1,
\end{cases}
\qquad
\left(\frac{2,F/\Q}{q}\right)=
\begin{cases}
1,&\left(\frac2q\right)_4=+1,\\
\gamma,&\left(\frac2q\right)_4=-1.
\end{cases} \end{array}$

On the other hand, since $K_{\fk P_1} \simeq \mathbb Q_2$, the proof of \cite[Proposition~8]{MR4940532} shows that 
 $$\left( \frac{2, F/\Q}{2} \right) =\left( \frac{2, \epsilon_D \sqrt{D}}{2} \right) = \left( \frac{2, \epsilon_D}{2} \right) \left( \frac{2, \sqrt{D}}{2} \right)= \begin{cases}
1, & \text{if } (\epsilon_D,2)_{\fk P_1}\kron{D}{2}_4=1 \\
\gamma, & \text{if } (\epsilon_D,2)_{\fk P_1}\kron{D}{2}_4=-1
\end{cases}$$
Putting everything together, Artin's product formula gives us that 
$$\kron{2}{p}_4 \kron{2}{q}_4 \kron{D}{2}_4 (\epsilon_D,2)_{\fk P_1} =1$$
which concludes the proof.

\end{proof}

On the other hand, we notice that $\epsilon_D^2 \in N(E(K_1))$. Therefore, if $\epsilon_D \not \in N(E(K_1))$, then $B_1 \simeq \Z/2\Z$ and $B_1'=B_1$.  Notice also that the maps $N:B_{n+1}\to B_n$ are surjective for all $n$.

\begin{lemma} \label{lem: property P_n}
Let $K$ be as above, and let $\epsilon_D$ be the fundamental unit of $K$.
Then
$$
(P_n)\ \forall n\ge 1
\iff
v_2(\log_2(\epsilon_D))=2
\iff
\epsilon_D\notin N_{K_1/K}(K_1^\times).
$$
\end{lemma}

\begin{proof}
Write
$$
2\mathcal O_K=\fk p_1\fk p_2.
$$
Since $2$ splits in $K$, we have
$$
K_{\fk p_1}\simeq K_{\fk p_2}\simeq \Q_2
$$ 
where $\Q_2$ represents the $2$-adic numbers.\\

Let $u_i\in \Q_2^\times$ be the image of $\epsilon_D$ in $K_{\fk p_i}$, and let
$\sigma$ be the non-trivial element of $\Gal(K/\mathbb Q)$. Since
$$
\sigma(\epsilon_D)=N_{K/\mathbb Q}(\epsilon_D)\,\epsilon_D^{-1},
$$
we get
$$
u_2=N_{K/\mathbb Q}(\epsilon_D)\,u_1^{-1}.
$$
As $\log_2(\pm 1)=0$, it follows that
$$
\log_2(u_2)=-\log_2(u_1),
$$
hence
$$
v_2(\log_2(u_1))=v_2(\log_2(u_2)).
$$
Thus the quantity $v_2(\log_2(\epsilon_D))$ is independent of the choice of dyadic place.

Set
$$
t:=v_2(\log_2(u_1))\in \Z_{\ge 2}\cup\{\infty\},
$$
where, as usual, $v_2(0)=\infty$.

For $u\in \Z_2^\times$, write
$$
u=\omega(u)\langle u\rangle,
\qquad
\omega(u)\in\{\pm1\},\quad \langle u\rangle\in 1+4\Z_2.
$$
Since
$$
\log_2:1+4\Z_2\longrightarrow 4\Z_2
$$
is an isomorphism of topological groups and $\log_2(x^{2^n})=2^n\log_2(x)$, we have
\begin{equation}\label{eq:2n-power-local}
u\in (\Z_2^\times)^{2^n}
\iff
\omega(u)=1\ \text{and}\ v_2(\log_2 u)\ge n+2.
\end{equation}

We first prove that
$$
(P_n)\ \forall n\ge 1
\iff
v_2(\log_2(\epsilon_D))=2.
$$

Assume first that $t=2$. Let
$$
x=(-1)^a\epsilon_D^m,
\qquad a\in\{0,1\},\quad 0\le m<2^n,
$$
represent a class in $\ker(\rho_n)$. Looking at the $\fk p_1$-component, the image of $x$ is
$$
(-1)^a u_1^m\in (\Z_2^\times)^{2^n}.
$$
By \eqref{eq:2n-power-local},
$$
v_2\!\bigl(\log_2((-1)^a u_1^m)\bigr)\ge n+2.
$$
Since $\log_2(-1)=0$, this becomes
$$
v_2(m\log_2(u_1))=v_2(m)+t\ge n+2.
$$
As $t=2$, we obtain
$$
v_2(m)\ge n,
$$
so $2^n\mid m$. Because $0\le m<2^n$, this forces $m=0$. Hence, $x=(-1)^a$.
But \eqref{eq:2n-power-local} also shows that $-1\notin (\Z_2^\times)^{2^n}$, so necessarily $a=0$.
Therefore $\ker(\rho_n)=0$; i.e. $(P_n)$ holds for every $n\ge 1$.

Conversely, suppose that $t>2$. If $t<\infty$, set
$$
n=t-1\ge 2;
$$
if $t=\infty$, choose any $n\ge 2$. Then the class of $\epsilon_D^2$ in
$$
E(K)/E(K)^{2^n}
$$
is non-trivial, since $\epsilon_D$ has infinite order and $2<2^n$. On the other hand, for each $i=1,2$ we have
$$
\omega(u_i^2)=1
$$
and
$$
v_2(\log_2(u_i^2))
=
v_2(2\log_2(u_i))
\ge n+2.
$$
Thus \eqref{eq:2n-power-local} gives
$$
u_i^2\in (\Z_2^\times)^{2^n}
\qquad (i=1,2).
$$
Hence, $\epsilon_D^2\in \ker(\rho_n)$, giving us a contradiction. Therefore, $t\le 2$.
Since always $t\ge 2$, we conclude that $t=2$.

This proves
$$
(P_n)\ \forall n\ge 1
\iff
v_2(\log_2(\epsilon_D))=2.
$$

We now prove that
$$
v_2(\log_2(\epsilon_D))=2
\iff
\epsilon_D\notin N_{K_1/K}(K_1^\times).
$$

Recall that
$$
\epsilon_D\in N_{K_1/K}(K_1^\times)
\iff
\epsilon_D\in N_{(K_1)_w/K_v}\bigl((K_1)_w^\times\bigr)
\quad\text{for every place }v\text{ of }K.
$$

At every finite place $v\nmid 2$, the extension $K_1/K$ is unramified or split, so every local unit is a norm; since $\epsilon_D$ is a unit, the local norm condition is automatic there. At the real places it is also automatic, because $2>0$. Hence only the dyadic places matter. Under the identifications
$$
K_{\fk p_i}\simeq \Q_2,
\qquad
(K_1)_{\fk P_i}\simeq \Q_2(\sqrt2),
$$
this becomes
$$
\epsilon_D\in N_{K_1/K}(K_1^\times)
\iff
u_i\in N_{\Q_2(\sqrt2)/\Q_2}\bigl(\Q_2(\sqrt2)^\times\bigr)
\qquad (i=1,2).
$$

By the local norm criterion for quadratic extensions,
$$
u_i\in N_{\Q_2(\sqrt2)/\Q_2}\bigl(\Q_2(\sqrt2)^\times\bigr)
\iff
(u_i,2)_2=1.
$$
Moreover,
$$
(u_2,2)_2
=
\bigl(N_{K/\Q}(\epsilon_D)u_1^{-1},2\bigr)_2
=
(u_1,2)_2,
$$
because $N_{K/\Q}(\epsilon_D)=\pm1$, $(\pm1,2)_2=1$, and $(u^{-1},2)_2=(u,2)_2$.
So it is enough to look at $u_1$.

For odd $u\in \Q_2^\times$, the explicit formula for the Hilbert symbol gives
$$
(u,2)_2=(-1)^{(u^2-1)/8}.
$$
Therefore,
$$
(u,2)_2=-1
\iff
u\equiv 3,5 \mod 8.
$$

On the other hand, if $u=\omega(u)\langle u\rangle$ with $\langle u\rangle\in 1+4\Z_2$, then
$$
\log_2(u)=\log_2(\langle u\rangle),
$$
and the series for $\log_2(1+x)$ with $x\in 4\Z_2$ shows that
$$
v_2(\log_2 u)=v_2(\langle u\rangle-1).
$$
Hence
$$
v_2(\log_2 u)=2
\iff
\langle u\rangle\equiv 5 \mod 8
\iff
u\equiv 3,5 \mod 8.
$$

Combining the last two equivalences, we obtain
$$
v_2(\log_2 u_1)=2
\iff
(u_1,2)_2=-1
\iff
u_1\notin N_{\Q_2(\sqrt2)/\Q_2}\bigl(\Q_2(\sqrt2)^\times\bigr).
$$
Since this is equivalent to
$$
\epsilon_D\notin N_{K_1/K}(K_1^\times),
$$
the proof is complete.
\end{proof}
\subsection{The subfield \texorpdfstring{$K_1$}{K1}}

The end goal of this section is to show that $A_1 \cong \Z / 2 \Z \times \Z / 2 \Z$. We begin with the following result about $F_0 = \Q(\sqrt{2D})$:

\begin{proposition}
    $A(F_0) \cong \Z / 2 \Z \times \Z / 2 \Z $ and $N(\epsilon_{2D})=1$. 
\end{proposition}

\begin{proof}
    From \cite[Proposition~B2]{MR309898}, since $l(D) = \kron{2}{p}_4 \kron{2}{q}_4 \kron{D}{2}_4= -1$, the $8$-rank of the narrow class group of $F_0$ is $0$.\\ 
    
    Since the the $4$-rank is equal to $1$ we must have $Cl_{2}(F_0)^{+} \cong \Z / 2 \Z \times \Z / 4 \Z$. Furthermore, from \cite[Corollary~1]{MR309898}, since the decomposition $2D= 2 \cdot D$ satisfies 
    $$\kron{2}{D}_4 \neq \kron{D}{2}_4,$$
    where we think of the symbol as multiplicative using Kaplan's notation, the equation $x^2-2Dy^2=-1$ has no solution and $N(\epsilon_{2D})=1$. Therefore, the narrow class group differs from the class group and $A(F_0) \cong \Z / 2 \Z \times \Z / 2 \Z $, since its $2$-rank is $2$ due to Theorem \ref{th: Genus theory}.
\end{proof}

From \cite[Theorem~3.2]{MR2783388}, since $$ \kron{2}{p}_4 \kron{2}{q}_4 \kron{D}{2}_4= -1,$$ one of the products $$ \kron{2}{p}_4  \kron{p}{2}_4,\kron{2}{q}_4 \kron{q}{2}_4$$ 
is equal to $-1$ and we must have $2$-rank$(A_n)=2 \; \forall n \geq 1$.  We finish with the following result:

\begin{lemma} \label{lem: A1=C2xC2}
    $A_1 \cong \Z / 2 \Z \times \Z / 2 \Z$, $\sqrt{\epsilon_{2D}} \in K_1$ and $\{ \epsilon_D, \sqrt{\epsilon_{2D}}, \epsilon_2 \}$ is a F.S.U. of $K_1$. 
\end{lemma}

\begin{proof}
    Let  $\epsilon_{2D} = r + s \sqrt{2D}$, where $r$ is odd and $s$ is even. Since $r^2-2Ds^2=1$, we have 
\begin{equation*}
    \begin{cases}
        r \pm 1 = d_{1} s_{1}^{2}, \\
        r \mp 1 = 2d_{2} s_{2}^{2},
    \end{cases}
\end{equation*}
with $d_{1}d_{2}=D$ and $s_{1}s_{2}=s$. Due to Azizi \cite{MR2296826}, $D(r \pm 1)$ is not a square and $d_1 \neq D$. \\

If $d_1=p$ or $d_1=q$ we have 
$$-1=\left(\frac{d_1}{d_2} \right)=\left(\frac{d_1 s_{1}^2}{d_2} \right)=\left(\frac{r \pm 1}{d_{2}} \right)=\left(\frac{r\mp1\pm2}{d_2} \right)= \left(\frac{\pm2}{d_{2}} \right)=1,$$
 which is impossible. We can conclude that $d_1=1$ and $d_2=D$.\\
 
 Let $z=s_1+s_2 \sqrt{2D}$, then $z^2=2 \epsilon_{2D}$, which implies $\sqrt{2 \epsilon_{2D}} \in F_0$. Since $\sqrt{2} \in K_1$, we must have $\sqrt{\epsilon_{2D}} \in K_1$. From Theorem \ref{kuroda}, since $\epsilon_2$ and $\epsilon_D$ have norm $-1$, we must have $\{ \epsilon_D, \sqrt{\epsilon_{2D}}, \epsilon_2 \}$ as a F.S.U. Therefore, $q(K_1) = [E_{K_1}:E(K)E(F_0)E(\mathbf Q_1)]=2$ and we have 
 $$|A_1|=\frac{q(K_1)|A_0||A(F_0)||A(\mathbf Q_1)|}{4}=\frac{2 \cdot 2 \cdot 4 \cdot 1}{4}=4.$$
 Since the $2$-rank is $2$, the result follows.
\end{proof}

\subsection{The subfield \texorpdfstring{$F_1$}{F1} and a bound on \texorpdfstring{$A_2$}{A2}}
We remember that $F_1 = \Q(\sqrt{(2+\sqrt{2})D}$ is the quadratic subextension of $K_2/\mathbf{Q}_1$ different from $K$ and $\mathbf Q_2$. We let $\alpha = 2+\sqrt{2} $ and $\delta = \alpha pq$. Then, $F_1=\mathbf Q_1(\sqrt{\delta})$. Hence, by Theorem \ref{th: lemmermeyer biquadratic class number}, 
    $$h(K_2)= \frac 18 q(K_2) h(F_1) h(\mathbf Q_2) h(K_1)$$

 In this section, we focus on computing $h(F_1)$. We let $A^{+}(L)$ denote the $2$-part of the narrow class group.

\begin{theorem} \label{th: r_4 of F1}
    Let $K=\Q(\sqrt{pq})$ with $p \equiv 1 \mod 8$, $q\equiv 9 \mod 16$ and $\kron{p}{q}=-1$. Then $r_4(A^+(F_1))=0$. Furthermore, if $\kron{2}{p}_4 \kron{p}{2}_4=-1$ or $\kron{2}{q}_4 \kron{q}{2}_4=-1$, then $r_2(A^+(F_1))=3$ and therefore $A^+(F_1)=(\Z/2\Z)^3$. \\
    
    In particular, this holds when $\kron{2}{p}\kron{2}{q}\kron{pq}{2}=-1$.  
\end{theorem}

\begin{proof}
    We start with the $2$-rank. We compute the generalized \Redei matrix $M_{F_1/\mathbf Q_1}$. The ramified primes in $F_1/\mathbf Q_1$ are the primes $\fk P$ above $2$, $\fk p, \bar{\fk p}$ above $p$ and $\fk q, \bar{\fk q}$ above $q$. 
\\ Write $$E(\mathbf Q_1)=<-1,\epsilon_2>$$ where $\epsilon_2=1+\sqrt{2}$ is the fundamental unit of $\mathbf Q_1$. 

\textbf{The $-1$ column:}
\\ $\bullet $ For $\fk p, \bar{\fk p}$, we have $\# \ot_F/\fk p = p \equiv 1 \mod 4$ and hence $(-1,\delta)_{\fk p}=(-1,\delta)_{\bar{\fk p}}=1$.
\\ $\bullet$ For $\fk q, \bar{\fk q}$, we have $\# \ot_F/\fk q = q \equiv 1 \mod 4$ and hence $(-1,\delta)_{\fk q_1}=(-1,\delta)_{\fk q_2}=1$.
\\ $\bullet$ For $\fk P$, by the product formula, we have $(-1,\delta)_{\fk P}=1$.
\\ With the identification of $0$ with $1$ and $1$ with $-1$, we conclude that the first column of the matrix is zero.

\textbf{The $\epsilon_2$ column:}\\
$\bullet $ For $\fk p, \bar{\fk p}$, the local symbol reduced to the quadratic residue symbol: 
\begin{align}
(\epsilon_2,\delta)_{\fk p} = \kron{\epsilon_2}{\fk p} = \kron{1 +\sqrt{2}}{p} \\
(\epsilon_2,\delta)_{\bar{\fk p}} = \kron{\epsilon_2}{\bar{\fk p}} = \kron{1- \sqrt{2}}{p}
\end{align}
Notice that $\kron{1 +\sqrt{2}}{p} \kron{1- \sqrt{2}}{p} = \kron{-1}{p}$=1 and therefore $$(\epsilon_2,\delta)_{\fk p_1} = (\epsilon_2,\delta)_{\fk p_2}$$
In fact, by Scholz's reciprocity law \cite[Prop~5.8, page~160]{MR1761696}, we have that $$\kron{1 +\sqrt{2}}{p} = \kron{2}{p}_4 \kron{p}{2}_4$$ and similarly for $q$. 
\\ $\bullet$ By the product formula, we have that $(\epsilon_2,\delta)_{\fk P}=1$.

In summary, the matrix becomes: 
$$\begin{array}{c} M_{E/F}\;=\;
\begin{pmatrix}
0 & 0\\
0 & \alpha_p\\
0 & \alpha_p\\
0 & \alpha_q\\
0 & \alpha_q
\end{pmatrix},
\qquad \alpha_p,\alpha_q\in\{0,1\} \end{array}$$

where $\alpha_p=0$ if $\kron{2}{p}_4 \kron{p}{2}_4=1$ and $\alpha_p=1$ otherwise, and similarly for $\alpha_q$. Therefore, if $\alpha_p=1$ or $\alpha_q=1$, then $r_2(A^+(F_1))=3$.

\textbf{Now for the 4-rank:}
By the product rule, the sum of all the rows of each column is zero. Therefore, the rank of the matrix is at most $4$. We show that it is exactly $4$ by showing that the first four rows corresponding to the primes $\fk p_1, \fk p_2, \fk q_1, \fk q_2$ are linearly independent.

Write $p = \pi \bar{\pi} $ and $q = \lambda \bar{\lambda}$ in $\mathbf Q_1$. Let $\Gal(\mathbf Q_1/\Q)=\{1,\tau \}$. Then clearly, $\tau(\fk q) = \bar{\fk q}$ and $\tau(\fk p) = \bar{\fk p}$. On the other hand, $\tau(\delta) = pq (2-\sqrt{2}) = \delta (1-\sqrt{2})^2$ and so 
$$\tau(\delta) \equiv \delta \in \mathbf{Q}_1^\times/(\mathbf{Q}_1^\times)^2$$ Therefore, the local Hilbert symbol behaves equivariently as 
$$(a,\delta)_{\fk q} = (\tau (a),\delta)_{\bar{\fk q}}$$ 
This already gives us that $$(\pi,\delta)_{\fk q} = (\bar{\pi},\delta)_{\bar{\fk q}}, \qquad (\lambda,\delta)_{\fk p} = (\bar{\lambda},\delta)_{\bar{\fk p}}$$

\begin{itemize}
    \item \textbf{Off-diagonal entries}: 
    At a non-dyadic prime $\fk t |\delta$, the extenstion $(F_1)_{\fk t}/(\mathbf Q_1)_{\fk t}$ is ramified. So if $u$ is a unit at $\fk t$, then $$(u,\delta)_{\fk t} = \kron{u}{\fk t}$$
Therefore in the $\fk q$- row, we have 
$$(\pi,\delta)_\fk q(\bar\pi,\delta)_\fk q
=\left(\frac{\pi}{\fk q}\right)\left(\frac{\bar\pi}{\fk q}\right)
=\left(\frac{\pi\bar\pi}{\fk q}\right)
=\left(\frac{p}{\fk q}\right)
=\left(\frac pq\right)=-1$$

Similarly, in the $\fk p$-row, we have
$$(\lambda,\delta)_\fk p(\bar\lambda,\delta)_\fk p
=\left(\frac{\lambda}{\fk p}\right)\left(\frac{\bar\lambda}{\fk p}\right)
=\left(\frac{\lambda\bar\lambda}{\fk p}\right)
=\left(\frac{q}{\fk p}\right)
=\left(\frac qp\right)=-1$$

\item \textbf{The diagonal entries:}
\\ Write $$\delta = pq(2+\sqrt{2})= \pi \bar{\pi} \lambda \bar \lambda(2+\sqrt{2}) = \lambda u_{\fk q}$$ where $u_{\fk q} \in \ot_{\mathbf Q_1,\fk q}^\times$  

Then 
\begin{align}
(\lambda,\delta)_{\fk q} = (u_\fk q, \delta)_{\fk q} & = \kron {u_\fk q}{\fk q}= \kron{\bar \lambda}{\fk q} \kron{p}{\fk q} \kron{2+\sqrt 2}{\fk q} \\ = \kron{\bar \lambda}{ \fk q} \kron{p}{q} \kron{2+\sqrt{2}}{\fk q} & = - \kron{\bar \lambda}{\fk q} \kron{2+\sqrt{2}}{\fk q}
\end{align}

Now, $$\kron{2+\sqrt{2}}{\fk q}= \kron{\sqrt{2}}{\fk q} \kron{1+\sqrt{2}}{\fk q} = \kron{2}{q}_4 \kron{q}{2}_4 \kron{2}{q}_4 = \kron{q}{2}_4 $$ 
with the last quantity being equal to $-1$ when $q \equiv 9 \mod 16$. Therefore $$(\lambda,\delta)_{\fk q} = \kron{\bar \lambda}{\fk q}$$

But since $\bar \lambda $ is a unit at $\fk q$, we also have have $(\bar \lambda, \delta)_{\fk q} = \kron{\bar \lambda}{\fk q}$. Therefore, $$(\lambda,\delta)_\fk q=(\bar\lambda,\delta)_\fk q$$

\end{itemize}
Overall, the matrix becomes 
$$ R_{F_1/\mathbf Q_1}\begin{pmatrix}
    a & b & c & d \\
    b & a & d & c \\
    e & f & g & g \\
    f & e & g & g \\
\end{pmatrix}
$$
with values in $\F_2$ and $c\neq d, e\neq f$. It has determinant $$\det(R)\equiv e^2c^2 +d^2 f^2 -e^2d^2-c^2f^2 = (c^2-d^2)(e^2-f^2) \equiv 1 \mod 2$$ 
Therefore, $\rank(R)=4$ and $r_4(A^+(F_1))=0$.

\end{proof}

\begin{remark}\label{rem:F1-upper-bound}
Under the product condition
$$
\left(\frac{D}{2}\right)_4
\left(\frac{2}{p}\right)_4
\left(\frac{2}{q}\right)_4=-1,
$$
the computation above shows that
$$
r_4(A^+(F_1))=0,\qquad r_2(A^+(F_1))=3.
$$
However, there are exactly $5$ ramified primes in $F_1/\mathbf Q_1$, namely: $$\fk P_2, \fk p, \bar{ \fk p}, \fk q, \bar{ \fk q}$$ Applying Chevalley's formula (Theorem \ref{th: Chevalley}) to the extension $F_1/\mathbf Q_1$, we get $$r_2(A(F_1))=3,$$ and therefore $$A(F_1)\simeq (\Z/2\Z)^3.$$

Consequently, applying the Kuroda--Lemmermeyer class number formula to
the biquadratic extension $K_2/\mathbf Q_1$, and using
$$
|A(K_1)|=4,\qquad |A(\mathbf Q_2)|=1,
$$
we obtain, on $2$-primary parts,
$$
|A(K_2)|
=
\frac{q(K_2)|A(K_1)||A(F_1)||A(Q_2)|}{8}
=
\frac{q(K_2)|A(F_1)|}{2}
=
4q(K_2).
$$
\end{remark}

\subsection{A Kumakawa-type capitulation argument}
Here we follow \cite{MR4262274} to exclude the possibility of $q(K_2)=2$. 

\begin{lemma}[Kumakawa-type capitulation argument]
\label{lem:kumakawa-capitulation}
Let $K_\infty/K$ be the cyclotomic $\Z_2$-extension, and let
$K_n$ denote its $n$-th layer. Let $A_n$ be the $2$-primary part of
the class group of $K_n$. Let $\fk q$ be a prime of $K$, and
suppose that
$$
A_0=\langle[\fk q]\rangle\simeq \Z/2\Z.
$$
Let $\fk q_1,\fk q_2$ be the primes of $K_2$ above
$\fk q$. Let $L_i$ be the maximal unramified abelian
$2$-extension of $K_i$. Assume that:
\begin{enumerate}
    \item $A_1\simeq \Z/2\Z\times \Z/2\Z$;
    \item $L_1\cap K_2=K_1$;
    \item each $\fk q_i$ splits completely in $L_1K_2/K_2$;
    \item $r_2(A_2)=2$ and $|A_2|\le 8$;
    \item the classes $a_i=[\fk q_i]\in A_2$ satisfy $a_i^2=1$.
\end{enumerate}
Then the image of $[\fk q]$ in $A_2$ is trivial.
\end{lemma}

\begin{proof}
This is the class-field-theoretic argument used by Kumakawa in the proof of
\cite[Theorem~2.1.5]{MR4262274}; we recall the short argument for completeness.

Since $A_1\simeq \Z/2\Z\times \Z/2\Z$ and $L_1\cap K_2=K_1$, we have
$$
\Gal(L_1K_2/K_2)\simeq \Gal(L_1/K_1)\simeq \Z/2\Z\times \Z/2\Z.
$$
Let
$$
a_i=[\fk q_i]\in A_2,\qquad i=1,2.
$$
The image of $[\fk q]$ in $A_2$ is
$$
[\fk q\mathcal O_{K_2}]
=
[\fk q_1\fk q_2]
=
a_1a_2.
$$
Thus, it suffices to prove $a_1a_2=1$.

Since each $\fk q_i$ splits completely in $L_1K_2/K_2$, class field
theory gives
$$
a_i\in
\ker\left(A_2\longrightarrow \Gal(L_1K_2/K_2)\right).
$$

By assumption, $r_2(A_2)=2$ and $|A_2|\le 8$. Hence,
$$
|A_2|\in\{4,8\}.
$$

If $|A_2|=4$, then
$$
A_2\simeq \Z/2\Z\times \Z/2\Z.
$$
The extension $L_1K_2/K_2$ is an unramified abelian $2$-extension of
degree $4$, so it must coincide with $L_2$. Since each $\fk q_i$
splits completely in $L_2/K_2$, its Artin symbol is trivial, and hence 
$$
a_1=a_2=1.
$$
Therefore $a_1a_2=1$.

Now suppose that $|A_2|=8$. Since $r_2(A_2)=2$, we have
$$
A_2\simeq \Z/4\Z\times \Z/2\Z.
$$
The Artin map
$$
A_2\longrightarrow \Gal(L_1K_2/K_2)
$$
is surjective. Its target has exponent $2$, so its kernel contains
$A_2^2$. Since $|A_2|=8$ and
$$
|\Gal(L_1K_2/K_2)|=4,
$$
the kernel has order $2$. But $A_2^2$ also has order $2$, so
$$
\ker\left(A_2\longrightarrow \Gal(L_1K_2/K_2)\right)=A_2^2.
$$
Thus,
$$
a_1,a_2\in A_2^2.
$$
The primes $\fk q_1$ and $\fk q_2$ are conjugate over $K$,
so $a_1$ and $a_2$ are conjugate under $\Gal(K_2/K)$. Since $A_2^2$
has order $2$, its unique nontrivial element is fixed by every automorphism.
Hence,
$$
a_1=a_2.
$$
By assumption $a_i^2=1$, so
$$
a_1a_2=a_1^2=1.
$$

In both cases, the image of $[\fk q]$ in $A_2$ is trivial.
\end{proof}

\begin{lemma}\label{lem:q less than 4 implies lambda 0}
Let $K=\Q(\sqrt{D})$ with $D=pq$, where
$$
p\equiv 1 \mod 8,\qquad q\equiv 9 \mod{16},\qquad \left(\frac{p}{q}\right)=-1.
$$
Assume moreover that
$$
\left(\frac{2}{p}\right)_4\left(\frac{2}{q}\right)_4\left(\frac{D}{2}\right)_4=-1
$$
and that
$$
q(K_2)<4.
$$
Then
$$
\lambda_2(K)=0.
$$
\end{lemma}

\begin{proof}
Let $\fk q$ be the unique prime of $K$ above $q$. By
Lemma~\ref{lem: A0=C2},
$$
A_0=\langle[\fk q]\rangle\simeq \Z/2\Z.
$$
Moreover, by Lemma~\ref{lem: property P_n}, the property $(P_n)$ holds
for all $n\ge 1$. By Theorem~\ref{th: greenberg split Bn}, it is enough to
show that $[\fk q]$ capitulates in $K_\infty$. It is enough to prove
that its image in $A_2$ is trivial.

Let $L_1$ be the maximal unramified abelian $2$-extension of $K_1$, and
let $L_2$ be the maximal unramified abelian $2$-extension of $K_2$. Since
$K_2/K_1$ is totally ramified at the dyadic primes while $L_1/K_1$ is
unramified, we have
$$
L_1\cap K_2=K_1.
$$
Therefore, by Lemma~\ref{lem: A1=C2xC2},
$$
\Gal(L_1K_2/K_2)
\simeq
\Gal(L_1/K_1)
\simeq
A_1
\simeq
\Z/2\Z\times \Z/2\Z.
$$

Let $\fk q_1,\fk q_2$ be the primes of $K_2$ above
$\fk q$. We claim that each $\fk q_i$ splits completely in
$L_1K_2/K_2$. This is the same splitting argument as in Kumakawa's proof
of \cite[Theorem~2.1.5]{MR4262274}. Indeed, let $F/K_1$ be any quadratic
subextension of $L_1/K_1$. Then $FK_2/K_1$ is biquadratic. Since
$q\equiv 9\mod {16}$, the primes of $K_1$ above $q$ are inert in
$K_2/K_1$, whereas $F/K_1$ is unramified.  Indeed, for a prime $\fk q\mid q$ of $\mathbf Q_1$, one has
$$
\left(\frac{\alpha}{\fk q}\right)
=
\left(\frac{q}{2}\right)_4
=
-1,
$$
because $q\equiv 9\mod {16}$. Hence the primes above $q$ are inert in
$\mathbf Q_2/\mathbf Q_1$, and therefore also in $K_2/K_1$. Hence, for a prime above $q$,
the decomposition group in $FK_2/K_1$ is cyclic of order $2$ and maps
nontrivially to $\Gal(K_2/K_1)$. Therefore its intersection with
$$
\Gal(FK_2/K_2)
$$
is trivial. This means that $\fk q_i$ splits completely in
$FK_2/K_2$. Since this holds for every quadratic subextension
$F/K_1$ of $L_1/K_1$, the prime $\fk q_i$ splits completely in
$L_1K_2/K_2$.

We now verify the remaining hypotheses of
Lemma~\ref{lem:kumakawa-capitulation}. By Remark~\ref{rem:F1-upper-bound},
$$
|A_2|= 4q(K_2).
$$
Since $q(K_2)<4$ and $q(K_2)$ is a power of $2$, we have
$$
q(K_2)\in\{1,2\},
$$
and hence
$$
|A_2|\le 8.
$$
On the other hand, by the $2$-rank computation in the cyclotomic tower,
$$
r_2(A_2)=2.
$$

Finally, if $a_i=[\fk q_i]\in A_2$, then
$$
a_i^2=1.
$$
Indeed, if $\fk Q_i$ is the prime of $\mathbf Q_2$ below
$\fk q_i$, then $\fk Q_i$ ramifies in $K_2/\mathbf Q_2$, so
$$
\fk Q_i\mathcal O_{K_2}=\fk q_i^2.
$$
Since $A(\mathbf Q_2)=1$, the class of $\fk Q_i\mathcal O_{K_2}$ is
trivial in $A_2$. Therefore $a_i^2=1$.

Lemma~\ref{lem:kumakawa-capitulation} now shows that the image of
$[\fk q]$ in $A_2$ is trivial. Thus, $[\fk q]$ capitulates in
$K_2$, and therefore in $K_\infty$. Since
$$
A_0=\langle[\fk q]\rangle,
$$
we obtain
$$
H_0=A_0.
$$
Applying Theorem~\ref{th: greenberg split Bn} together with
Lemma~\ref{lem: property P_n}, we conclude that
$$
\lambda_2(K)=0.
$$
\end{proof}

\subsection{Computing the index}

Throughout this subsection put
$$
D=pq, \quad \mathbf Q_1=\Q(\sqrt 2), \quad \alpha=2+\sqrt 2,
$$
so that
$$
K_1=\mathbf Q_1(\sqrt D), \quad F_1=\mathbf Q_1(\sqrt{\alpha D}), \quad \mathbf Q_2=\mathbf Q_1(\sqrt \alpha), \quad K_2=K_1F_1=K_1\mathbf Q_2=F_1 \mathbf Q_2.
$$
We assume
$$
p \equiv 1 \mod 8, \quad q \equiv 9 \mod {16}, \quad \kron{p}{q}=-1,
$$
and, when it is needed, the two hypotheses
$$
\kron{2}{p}_4\kron{2}{q}_4\kron{D}{2}_4=-1
$$
and
$$
\kron{2}{p}_4=-1 \quad \text{or} \quad \kron{2}{q}_4=-1.
$$
Let
$$
H=E(K_1)E(F_1)E(\mathbf Q_2), \quad S=\frac{H \cap K_2^{\times 2}}{H^2}=\frac{E(K_2)^2}{H^2}.
$$ 
Thus $q(K_2)=[E(K_2):H]$ is the Hasse unit index attached to the biquadratic extension $K_2/\mathbf Q_1$.

\begin{lemma}\label{lem:wada-description}
With the notation above,
$$
q(K_2)=[E(K_2):H]
=
\left|\frac{H\cap K_2^{\times 2}}{H^2}\right|
=
|S|.
$$
\end{lemma}

\begin{proof}
Since $K_2/\mathbf Q_1$ is biquadratic with quadratic subfields
$K_1,F_1,\mathbf Q_2$, for every $u\in E(K_2)$ we have
\[
u^2
=
\frac{
N_{K_2/K_1}(u)\,N_{K_2/F_1}(u)\,N_{K_2/\mathbf Q_2}(u)
}{
N_{K_2/\mathbf Q_1}(u)
}.
\]
Notice that  $E(K_2)^2\subset H$ and hence
\[
\frac{H\cap K_2^{\times 2}}{H^2}
=
\frac{E(K_2)^2}{H^2}.
\]

Since $K_2$ is totally real, we have
\[
E(K_2)[2]=\{\pm 1\}\subset H.
\]
Therefore, the squaring map induces an isomorphism
\[
E(K_2)/H \simeq E(K_2)^2/H^2.
\]

Thus,
\[
E(K_2)/H\cong E(K_2)^2/H^2
=
\frac{H\cap K_2^{\times 2}}{H^2}.
\]
Taking cardinalities gives
\[
q(K_2)=[E(K_2):H]
=
\left|\frac{H\cap K_2^{\times 2}}{H^2}\right|
=
|S|.
\]
\end{proof}

\begin{lemma}\label{lem:unit-normalizations-K1-Q2}
There are units $u_0,u_1 \in E(K_1)$ and $w_0,w_1 \in E(\mathbf Q_2)$ such that
$$
u_0=\epsilon_D, \quad N_{K_1/\mathbf Q_1}(u_0)=-1, \quad N_{K_1/\mathbf Q_1}(u_1)=1,
$$
and
$$
N_{\mathbf Q_2/\mathbf Q_1}(w_0)=-1, \quad N_{\mathbf Q_2/\mathbf Q_1}(w_1)=\epsilon_2.
$$
Moreover
$$
E(K_1)=\langle -1,u_0,u_1,\epsilon_2\rangle, \quad E(\mathbf Q_2)=\langle -1,w_0,w_1,\epsilon_2\rangle.
$$ 
\end{lemma}

\begin{proof}
By the unit computation for $K_1$, we have $\eta=\sqrt{\epsilon_{2D}} \in K_1$ and
$$
E(K_1)=\langle -1,\epsilon_D,\eta,\epsilon_2\rangle.
$$
Set $u_0=\epsilon_D$. Since $N_{K_1/\mathbf Q_1}(u_0)=-1$, we choose
$$
u_1=\begin{cases}
\eta, & N_{K_1/\mathbf Q_1}(\eta)=1,\\
\eta\epsilon_D, & N_{K_1/\mathbf Q_1}(\eta)=-1.
\end{cases}
$$
Then $N_{K_1/\mathbf Q_1}(u_1)=1$, and $u_0,u_1,\epsilon_2$ still give a fundamental system of units of $K_1$.

For $\mathbf Q_2=\mathbf Q_1(\sqrt \alpha)$, one may take
$$
w_1=1+\sqrt 2+\sqrt{2+\sqrt 2}, \quad w_0=1+\sqrt 2+\sqrt 2\sqrt{2+\sqrt 2}.
$$
A direct calculation gives
$$
N_{\mathbf Q_2/\mathbf Q_1}(w_1)=\epsilon_2, \quad N_{\mathbf Q_2/\mathbf Q_1}(w_0)=-1,
$$
and $\{w_0,w_1,\epsilon_2\}$ is a fundamental system of units of $\mathbf Q_2$.
\end{proof}

\begin{lemma}\label{lem:base-units-square-classes}
For $L=K_1,F_1,\mathbf Q_2,K_2$, the four classes
$$
1, \quad -1, \quad \epsilon_2, \quad -\epsilon_2
$$
are distinct in $L^\times/L^{\times 2}$.
In particular, if $a \in E(\mathbf Q_1)$ and $a \in L^{\times 2}$, then $a \in E(\mathbf Q_1)^2$.
\end{lemma}

\begin{proof}
All fields under consideration are totally real, so $-1$ is not a square. Also $\epsilon_2=1+\sqrt 2$ is not totally positive, since its conjugate $1-\sqrt 2$ is negative. Similarly $-\epsilon_2$ is negative under the identity embedding of $Q_1$. Hence none of $-1,\epsilon_2,-\epsilon_2$ is a square in any of these fields.

The last assertion follows because
$$
E(\mathbf Q_1)/E(\mathbf Q_1)^2=\{1,-1,\epsilon_2,-\epsilon_2\}.
$$
\end{proof}

\begin{lemma}\label{lem:no-new-single-field-squares}
Under the standing assumptions of this subsection, one has
$$
E(K_1)\cap K_2^{\times 2}=E(K_1)^2,
\qquad
E(\mathbf Q_2)\cap K_2^{\times 2}=E(\mathbf Q_2)^2.
$$
\end{lemma}

\begin{proof}
Since $K_2=K_1(\sqrt \alpha)$, the quadratic square criterion gives, for $u \in K_1^\times$,
$$
u \in K_2^{\times 2} \iff u \in K_1^{\times 2} \text{ or } u \in \alpha K_1^{\times 2}.
$$
If $u \in E(K_1)$, the second possibility is impossible: $\alpha$ has odd valuation at the primes of $K_1$ above $2$, whereas $u$ is a unit. Hence
$$
E(K_1) \cap K_2^{\times 2}=E(K_1)^2.
$$

Similarly, $K_2=\mathbf Q_2(\sqrt D)$. Thus, for $u \in \mathbf Q_2^\times$,
$$
u \in K_2^{\times 2} \iff u \in \mathbf Q_2^{\times 2} \text{ or } u \in D\mathbf Q_2^{\times 2}.
$$
If $u \in E(\mathbf Q_2)$, the second possibility is impossible because $D=pq$ has odd valuation at primes above $p$ and $q$. Therefore,
$$
E(\mathbf Q_2) \cap K_2^{\times 2}=E(\mathbf Q_2)^2.
$$
\end{proof}

\begin{lemma}\label{lem:no-new-Q2-mixed-squares}
Under the standing assumptions of this subsection, one has
$$
E(K_1)E(\mathbf Q_2)\cap K_2^{\times 2}
=
E(K_1)^2E(\mathbf Q_2)^2.
$$
If, moreover,
$$
E(F_1)\cap K_2^{\times 2}=E(F_1)^2,
$$
then
$$
E(F_1)E(\mathbf Q_2)\cap K_2^{\times 2}
=
E(F_1)^2E(\mathbf Q_2)^2.
$$
\end{lemma}

\begin{proof}
We prove the first equality. Let $x \in E(K_1)E(\mathbf Q_2) \cap K_2^{\times 2}$. Since $-1$ and $\epsilon_2$ lie in $E(K_1)$, we may write, modulo $E(K_1)^2E(\mathbf Q_2)^2$,
$$
x \equiv u w_0^a w_1^b, \quad u \in E(K_1), \quad a,b \in \mathbf F_2.
$$
Taking norms from $K_2$ to $K_1$ gives
$$
N_{K_2/K_1}(x)=u^2(-1)^a\epsilon_2^b.
$$
This is a square in $K_1$. By Lemma~\ref{lem:base-units-square-classes}, we get $a=b=0$. Hence $x \equiv u$, and then Lemma~\ref{lem:no-new-single-field-squares} gives that $x$ is trivial modulo $E(K_1)^2E(\mathbf Q_2)^2$.

The second equality is identical. If $x \in E(F_1)E(\mathbf Q_2) \cap K_2^{\times 2}$, write $x \equiv v w_0^a w_1^b$ with $v \in E(F_1)$ and take norms from $K_2$ to $F_1$.
\end{proof}

\begin{lemma}\label{lem:K1F1-one-class}
Assume
$$
E(F_1) \cap K_2^{\times 2}=E(F_1)^2.
$$
Set
$$
S_{K_1F_1}=\frac{E(K_1)E(F_1) \cap K_2^{\times 2}}{E(K_1)^2E(F_1)^2}.
$$
Then $\dim_{\mathbf F_2} S_{K_1F_1} \le 1$.
Moreover, if $S_{K_1F_1}$ is nontrivial, its unique nontrivial class has a representative
$$
A= u_1 v \epsilon_2^e
$$
with $e \in \mathbf F_2$, $v \in E(F_1)$, and $N_{F_1/\mathbf Q_1}(v)=1$.
\end{lemma}

\begin{proof}
Let $x \in E(K_1)E(F_1) \cap K_2^{\times 2}$. Write
$$
x=kf, \quad k \in E(K_1), \quad f \in E(F_1).
$$
Modulo $E(K_1)^2$ we may write
$$
k=(-1)^r u_0^a u_1^b\epsilon_2^c, \quad r,a,b,c \in \mathbf F_2.
$$

Since $-1 \in E(F_1)$, replacing $f$ by $(-1)^r f$ allows us to absorb the factor $(-1)^r$ into $f$. Thus, without changing $x$, we may write
\[
k=u_0^a u_1^b\epsilon_2^c.
\]

Taking norms from $K_2$ to $K_1$ gives
$$
N_{K_2/K_1}(x)=k^2N_{F_1/\mathbf Q_1}(f).
$$
Since $x$ is a square in $K_2$, this norm is a square in $K_1$. Thus $N_{F_1/\mathbf Q_1}(f) \in E(\mathbf Q_1) \cap K_1^{\times 2}$. By Lemma~\ref{lem:base-units-square-classes},
$$
N_{F_1/\mathbf Q_1}(f) \in E(\mathbf Q_1)^2.
$$
Choose $\beta \in E(\mathbf Q_1)$  with $N_{F_1/\mathbf Q_1}(f)=\beta^2$. Replacing $f$ by $f/\beta$ and $k$ by $k\beta$ does not change $x$, and lets us assume
$$
N_{F_1/\mathbf Q_1}(f)=1.
$$
Now take norms from $K_2$ to $\mathbf Q_2$. We get
$$
N_{K_2/\mathbf Q_2}(x)=N_{K_1/\mathbf Q_1}(k)N_{F_1/\mathbf Q_1}(f) \cdot \text{square}=(-1)^a \cdot \text{square}.
$$
This is a square in $\mathbf Q_2$, so Lemma~\ref{lem:base-units-square-classes} gives $a=0$.
Therefore every class in $S_{K_1F_1}$ has a representative
$$
u_1^b f\epsilon_2^c
$$
with $N_{F_1/\mathbf Q_1}(f)=1$.

If $b=0$, then the representative lies in $E(F_1) \cap K_2^{\times 2}$, hence it is trivial by assumption. Thus, any nontrivial class must have $b=1$.

Finally, suppose $x_1$ and $x_2$ are two nontrivial classes. By the previous paragraph, they have representatives
$$
x_i=u_1f_i\epsilon_2^{c_i}, \quad N_{F_1/\mathbf Q_1}(f_i)=1.
$$
Then,
$$
x_1x_2=f_1f_2\epsilon_2^{c_1+c_2} \in E(F_1) \cap K_2^{\times 2}.
$$
By the assumption on $E(F_1)$, this product is trivial modulo $E(F_1)^2$, hence $x_1=x_2$ in $S_{K_1F_1}$. Therefore, $\dim_{\mathbf F_2} S_{K_1F_1} \le 1$.
\end{proof}

\begin{lemma}\label{lem:outside-normal-form}
Assume
$$
E(F_1) \cap K_2^{\times 2}=E(F_1)^2.
$$
Let $T$ be the set of classes in $S$ which do not lie in the image of $S_{K_1F_1}$. If $B \in T$, then $B$ has a representative of the form
$$
B= u_0 u_1^b f w_0\epsilon_2^d
$$
with $s,b,d \in \mathbf F_2$, $f \in E(F_1)$, and
$$
N_{F_1/\mathbf Q_1}(f)=-1.
$$
In particular, whenever $T \ne \varnothing$, one has
$$
-1 \in N_{F_1/Q_1}(E(F_1)).
$$
\end{lemma}

\begin{proof}
Let $B \in T$ and choose a representative $x \in H \cap K_2^{\times 2}$. Write
$$
x=kfq, \quad k \in E(K_1), \quad f \in E(F_1), \quad q \in E(\mathbf Q_2).
$$
Since $-1$ and $\epsilon_2$ belong to both $E(K_1)$ and $E(F_1)$, and since $E(\mathbf Q_2)=\langle -1,w_0,w_1,\epsilon_2\rangle$, we may write, modulo $H^2$,
$$
k=(-1)^r u_0^a u_1^b\epsilon_2^c, \quad q=w_0^m w_1^n,
$$
with all exponents in $\mathbf F_2$.

Taking norms from $K_2$ to $K_1$ gives
$$
N_{F_1/\mathbf Q_1}(f)(-1)^m\epsilon_2^n \in E(\mathbf Q_1)^2.
$$
Taking norms from $K_2$ to $\mathbf Q_2$ gives
$$
(-1)^aN_{F_1/\mathbf Q_1}(f) \in E(\mathbf Q_1)^2.
$$
Dividing these two relations, we obtain
$$
(-1)^{a+m}\epsilon_2^n \in E(\mathbf Q_1)^2.
$$
By Lemma~\ref{lem:base-units-square-classes}, this forces
$$
n=0, \quad a=m.
$$
If $m=0$, then $q$ is trivial modulo $E(\mathbf Q_2)^2E(\mathbf Q_1)$, so the class of $x$ lies in the image of $S_{K_1F_1}$, contrary to $B \in T$. Hence $m=1$, and therefore $a=1$.

The first norm relation now gives
$$
N_{F_1/\mathbf Q_1}(f) \in -E(\mathbf Q_1)^2.
$$
Choose $\beta \in E(\mathbf Q_1)$ such that
$$
N_{F_1/\mathbf Q_1}(f)=-\beta^2.
$$
Replacing $f$ by $f/\beta$ and absorbing $\beta$ into the $(-1)^s\epsilon_2^d$ factor gives the desired representative
$$
B=(-1)^s u_0u_1^b f w_0\epsilon_2^d
$$
with $N_{F_1/\mathbf Q_1}(f)=-1$, where again you can omit $(-1)^s$
\end{proof}

\begin{lemma}[A norm obstruction]\label{lem:epsilonD-not-norm-K2K1}
Assume
$$
\kron{2}{p}_4\kron{2}{q}_4\kron{D}{2}_4=-1.
$$
Then
$$
\epsilon_D \notin N_{K_2/K_1}(K_2^\times).
$$
\end{lemma}

\begin{proof}
By Lemma~\ref{lem: ed is norm iff}, the quartic-symbol condition implies
$$
\epsilon_D \notin N_{K_1/K}(K_1^\times).
$$
Equivalently, for some dyadic prime $\fk p \mid 2$ of $K$,
$$
(\epsilon_D,2)_{\fk p}=-1.
$$
Let $\fk P$ be a prime of $K_1$ above $\fk p$. Locally,
$$
K_{1,\fk P}=K_{\fk p}(\sqrt 2), \quad K_{2,\fk P}=K_{1,\fk P}(\sqrt \alpha).
$$
Moreover
$$
N_{K_{1,\fk P}/K_{\fk p}}(\alpha)=(2+\sqrt 2)(2-\sqrt 2)=2.
$$
By the projection formula for Hilbert symbols,
$$
(\epsilon_D,\alpha)_{\fk P}=(\epsilon_D,N_{K_{1,\fk P}/K_{\fk p}}(\alpha))_{\fk p}=(\epsilon_D,2)_{\fk p}=-1.
$$
Thus $\epsilon_D$ is not a local norm from $K_{2,\fk P}$ to $K_{1,\fk P}$, and therefore it is not a global norm from $K_2$ to $K_1$.
\end{proof}

\begin{lemma}\label{lem:T-empty-if-SK1F1-nontrivial}
Assume
$$
E(F_1) \cap K_2^{\times 2}=E(F_1)^2
$$
and
$$
\kron{2}{p}_4\kron{2}{q}_4\kron{D}{2}_4=-1.
$$
If $S_{K_1F_1}$ is nontrivial, then $T=\varnothing$.
\end{lemma}

\begin{proof}
Assume $S_{K_1F_1}$ is nontrivial. By Lemma~\ref{lem:K1F1-one-class}, its nontrivial class has a representative
$$
A= u_1v\epsilon_2^e
$$
with $e \in \mathbf F_2$, $v \in E(F_1)$, and $N_{F_1/Q_1}(v)=1$.

Suppose, for contradiction, that $T \ne \varnothing$. By Lemma~\ref{lem:outside-normal-form}, we can choose a representative
$$
B= u_0u_1^b f w_0\epsilon_2^d
$$
with $b,d \in \mathbf F_2$, $f \in E(F_1)$, and $N_{F_1/Q_1}(f)=-1$.

Since $A$ and $B$ are squares in $K_2$, the product $BA^b$ is also a square. Dividing by the square $u_1^{2b}$, we get
$$
C=u_0f v^b w_0\epsilon_2^{d+be} \in K_2^{\times 2}.
$$
Write $C=y^2$ with $y \in K_2^\times$.

Now take norms from $K_2$ to $K_1$. Since $u_0 \in K_1$, $f,v \in F_1$, and $w_0 \in Q_2$, we have
$$
N_{K_2/K_1}(C)=u_0^2N_{F_1/\mathbf Q_1}(f)N_{F_1/\mathbf Q_1}(v)^bN_{\mathbf Q_2/\mathbf Q_1}(w_0)\epsilon_2^{2(d+be)}.
$$
Using $N_{F_1/\mathbf Q_1}(f)=-1$, $N_{F_1/\mathbf Q_1}(v)=1$, and $N_{\mathbf Q_2/\mathbf Q_1}(w_0)=-1$, this becomes
$$
N_{K_2/K_1}(C)=u_0^2\epsilon_2^{2(d+be)}=(u_0\epsilon_2^{d+be})^2.
$$
Since $C=y^2$, we get
$$
N_{K_2/K_1}(y)^2=(u_0\epsilon_2^{d+be})^2,
$$
so
$$
N_{K_2/K_1}(y)=\pm u_0\epsilon_2^{d+be}.
$$
But $-1$ and $\epsilon_2$ are norms from $K_2$ to $K_1$ because
$$
-1=N_{K_2/K_1}(w_0), \quad \epsilon_2=N_{K_2/K_1}(w_1).
$$
Therefore $u_0=\epsilon_D$ is a norm from $K_2$ to $K_1$, contradicting Lemma~\ref{lem:epsilonD-not-norm-K2K1}. Hence $T=\varnothing$.
\end{proof}

\begin{corollary}\label{cor:qK2-le2-from-EF1}
Assume
$$
E(F_1) \cap K_2^{\times 2}=E(F_1)^2
$$
and
$$
\kron{2}{p}_4\kron{2}{q}_4\kron{D}{2}_4=-1.
$$
Then
$$
q(K_2) \le 2.
$$
\end{corollary}

\begin{proof}
If $S_{K_1F_1}$ is nontrivial, then Lemma~\ref{lem:T-empty-if-SK1F1-nontrivial} gives $T=\varnothing$. Hence
$$
S=S_{K_1F_1}
$$
and Lemma~\ref{lem:K1F1-one-class} gives $|S| \le 2$.

Now suppose $S_{K_1F_1}$ is trivial. We show that $|T| \le 1$. Let $B_1,B_2 \in T$. By Lemma~\ref{lem:outside-normal-form}, we may write
$$
B_i=u_0u_1^{b_i}f_iw_0\epsilon_2^{d_i}, \quad N_{F_1/Q_1}(f_i)=-1.
$$
Then, in $S$,
$$
B_1B_2=u_1^{b_1+b_2}f_1f_2\epsilon_2^{d_1+d_2},
$$
because $u_0^2$ and $w_0^2$ are in $H^2$. Thus, $B_1B_2$ lies in the image of $S_{K_1F_1}$. Since $S_{K_1F_1}$ is trivial, $B_1B_2=1$ in $S$, and hence $B_1=B_2$. Therefore $|T| \le 1$.

In both cases $S=S_{K_1F_1} \cup T$ has at most two elements. By Lemma~\ref{lem:wada-description},
$$
q(K_2)=|S| \le 2.
$$
\end{proof}

\begin{lemma}\label{lem:F1-no-new-squares-criterion}
Let $\fk P_2=(\sqrt 2)$ be the dyadic prime of $\mathbf Q_1$, and let $\fk l$ be the unique prime of $F_1$ above $\fk P_2$. Then the following are equivalent:
\begin{enumerate}
\item $E(F_1) \cap K_2^{\times 2}=E(F_1)^2$.
\item $\fk l$ is not principal.
\item Neither $\sqrt 2$ nor $\alpha$ is a norm from $F_1^\times$ to $\mathbf Q_1^\times$.
\end{enumerate}
Moreover,
$$
\sqrt 2 \in N_{F_1/\mathbf Q_1}(F_1^\times) \iff \kron{2}{p}_4=\kron{2}{q}_4=1,
$$
and
$$
\alpha \in N_{F_1/\mathbf Q_1}(F_1^\times) \iff \kron{p}{2}_4=\kron{q}{2}_4=1.
$$
Consequently, since $q \equiv 9 \mod {16}$,
$$
E(F_1) \cap K_2^{\times 2}=E(F_1)^2 \iff \kron{2}{p}_4=-1 \quad \text{or} \quad \kron{2}{q}_4=-1.
$$
\end{lemma}

\begin{proof}
First note that $-1$ is a norm from $F_1$ to $\mathbf Q_1$. Indeed, the computation of the $-1$ column in Theorem~\ref{th: r_4 of F1} gives
$$
(-1,\delta)_v=1
$$
at every ramified finite place $v$ of $\mathbf Q_1$, and the local norm condition is automatic at the remaining finite places and at the real places. Hence, the Hasse norm theorem gives
$$
-1 \in N_{F_1/\mathbf Q_1}(F_1^\times).
$$

Since $K_2=F_1(\sqrt \alpha)$, the quadratic square criterion gives, for every $u \in F_1^\times$,
$$
u \in K_2^{\times 2} \iff u \in F_1^{\times 2} \text{ or } u \in \alpha F_1^{\times 2}.
$$
Thus, a nontrivial class in $(E(F_1) \cap K_2^{\times 2})/E(F_1)^2$ exists if and only if there are $u \in E(F_1)$ and $x \in F_1^\times$ such that
$$
u=\alpha x^2.
$$
Now $\alpha=\epsilon_2\sqrt 2$ and $\fk P_2$ ramifies in $F_1/\mathbf Q_1$, so
$$
\fk P_2\mathcal O_{F_1}=\fk l^2, \quad (\alpha)=(\sqrt 2)\mathcal O_{F_1}=\fk l^2.
$$
Therefore,
$$
1=(u)=(\alpha)(x)^2=\fk l^2(x)^2,
$$
so $(x)=\fk l^{-1}$. Hence, such a nontrivial unit square class exists if and only if $\fk l$ is principal. This proves $(1) \Longleftrightarrow (2)$.

Assume next that $\fk l$ is principal, say $\fk l=(x)$. Then
$$
(N_{F_1/\mathbf Q_1}(x))=N_{F_1/\mathbf Q_1}(\fk l)=\fk P_2=(\sqrt 2),
$$
so
$$
N_{F_1/\mathbf Q_1}(x)=\epsilon\sqrt 2
$$
for some $\epsilon \in E(\mathbf Q_1)$. Since $E(\mathbf Q_1)=\langle -1,\epsilon_2\rangle$, since $-1$ is a norm, and since every square in $E(\mathbf Q_1)$ is a norm, the unit $\epsilon$ is congruent modulo norms to either $1$ or $\epsilon_2$. Thus, either $\sqrt 2$ or $\epsilon_2\sqrt 2=\alpha$ is a norm from $F_1$ to $\mathbf Q_1$.

Conversely, suppose $\sqrt 2$ or $\alpha$ is a norm from $F_1$ to $\mathbf Q_1$. Choose $x \in F_1^\times$ with
$$
N_{F_1/\mathbf Q_1}(x)=\sqrt 2 \quad \text{or} \quad N_{F_1/\mathbf Q_1}(x)=\alpha.
$$
Then
$$
N_{F_1/\mathbf Q_1}((x)\fk l^{-1})=1.
$$
By Hilbert 90 for ideals, there is a fractional ideal $\fk a$ of $F_1$ such that
$$
(x)\fk l^{-1}=\fk a^{1-\sigma},
$$
where $\sigma$ is the nontrivial element of $\operatorname{Gal}(F_1/\mathbf Q_1)$. Passing to the class group gives
$$
[\fk l]=[\fk a]^{\sigma-1}.
$$
Since $h(\mathbf Q_1)=1$, the action of $\sigma$ on $\operatorname{Cl}(F_1)$ is inversion. Hence,
$$
[\fk l]=[\fk a]^{-2}.
$$
Thus, $[\fk l]$ is a square in the class group. On the other hand, $\fk l^2=(\alpha)$, so $[\fk l]$ has order dividing $2$. By Theorem~\ref{th: r_4 of F1}, $r_4(A(F_1))=0$, so no nontrivial element of $A(F_1)[2]$ is a square. Therefore $[\fk l]=1$, and $\fk l$ is principal. This proves $(2) \Longleftrightarrow (3)$.

It remains to evaluate the two norm conditions. Since $F_1=\mathbf Q_1(\sqrt \delta)$, the Hasse norm theorem gives
$$
\beta \in N_{F_1/\mathbf Q_1}(F_1^\times) \iff (\beta,\delta)_v=1 \quad \text{for every place } v \text{ of } \mathbf Q_1.
$$
The only finite ramified places are $\fk P_2$, the two primes $\fk p,\overline{\fk p}$ above $p$, and the two primes $\fk q,\overline{\fk q}$ above $q$. The dyadic symbol is determined by the product formula, and the two symbols above a fixed rational prime agree. Hence it suffices to check one prime above $p$ and one prime above $q$.

For $\sqrt 2$,
$$
(\sqrt 2,\delta)_{\fk p}=\kron{\sqrt 2}{\fk p}=\kron{2}{p}_4, \quad (\sqrt 2,\delta)_{\fk q}=\kron{\sqrt 2}{\fk q}=\kron{2}{q}_4.
$$
Therefore
$$
\sqrt 2 \in N_{F_1/\mathbf Q_1}(F_1^\times) \iff \kron{2}{p}_4=\kron{2}{q}_4=1.
$$

For $\alpha=\epsilon_2\sqrt 2$, multiplicativity gives
$$
(\alpha,\delta)_{\fk p}=(\epsilon_2,\delta)_{\fk p}(\sqrt 2,\delta)_{\fk p}, \quad (\alpha,\delta)_{\fk q}=(\epsilon_2,\delta)_{\fk q}(\sqrt 2,\delta)_{\fk q}.
$$
By the computation of the $\epsilon_2$ column in Theorem~\ref{th: r_4 of F1},
$$
(\epsilon_2,\delta)_{\fk p}=\kron{2}{p}_4\kron{p}{2}_4, \quad (\epsilon_2,\delta)_{\fk q}=\kron{2}{q}_4\kron{q}{2}_4.
$$
Thus
$$
(\alpha,\delta)_{\fk p}=\kron{p}{2}_4, \quad (\alpha,\delta)_{\fk q}=\kron{q}{2}_4,
$$
and therefore
$$
\alpha \in N_{F_1/\mathbf Q_1}(F_1^\times) \iff \kron{p}{2}_4=\kron{q}{2}_4=1.
$$
Since $q \equiv 9 \mod {16}$, we have $\kron{q}{2}_4=-1$. Hence, $\alpha$ is not a norm. The criterion therefore reduces to
$$
E(F_1) \cap K_2^{\times 2}=E(F_1)^2 \iff \kron{2}{p}_4=-1 \quad \text{or} \quad \kron{2}{q}_4=-1.
$$
\end{proof}

\begin{proposition}\label{prop:qK2-le2}
Under the hypotheses of Theorem~1.1,
$$
q(K_2) \le 2.
$$
\end{proposition}

\begin{proof}
By Lemma~\ref{lem:F1-no-new-squares-criterion} and the hypothesis
$$
\kron{2}{p}_4=-1 \quad \text{or} \quad \kron{2}{q}_4=-1,
$$
we have
$$
E(F_1) \cap K_2^{\times 2}=E(F_1)^2.
$$
The product condition
$$
\kron{2}{p}_4\kron{2}{q}_4\kron{D}{2}_4=-1
$$
then allows us to apply Corollary~\ref{cor:qK2-le2-from-EF1}. Hence
$$
q(K_2) \le 2.
$$
\end{proof}

\begin{proof}[Proof of Theorem~1.1]
By Proposition~\ref{prop:qK2-le2},
$$
q(K_2) \le 2 < 4.
$$
Therefore Lemma~\ref{lem:q less than 4 implies lambda 0} gives
$$
\lambda_2(K)=0.
$$
\end{proof}

\begin{remark}
We finish by comparing our argument with the work of Kumakawa \cite{MR4262274}.
In \cite{MR4262274}, the author studies the case
\[
K=\Q(\sqrt{pq}), \qquad
p \equiv 3 \mod 8, \qquad q\equiv 9 \mod {16}, \qquad
\left(\frac{q}{p}\right)=-1, \qquad
\left(\frac{2}{q}\right)_4=1 .
\]
They prove the following result.
\begin{theorem*}[Kumakawa]
With the above assumptions, if $r_4(A_2)=1$, then
\[
|A_n|=|A_2|
\]
for all $n \geq 2$. Therefore, $\lambda_2(K)=0$.
\end{theorem*}

The proof in \cite{MR4262274} and the proof of Theorem \ref{th: main th}
have the same general flavor. In both cases, the key point is to force the
capitulation of the class in $A_0$ by analyzing the second layer $K_2$ and
then applying Greenberg's criterion. 

There is, however, an important difference between the two arguments. In
Kumakawa's setting, the necessary bound on the Hasse unit index is essentially
automatic. Equivalently, the analogue of the condition
\[
q(K_2)<4
\]
does not require a separate computation. In our setting this was the main
extra difficulty: the reduction to capitulation only gives a useful conclusion
after proving that the Hasse unit index of the biquadratic extension
$K_2/\mathbf Q_1$ is small. This is why a substantial part of Section $7$
is devoted to a square-class computation of
\[
E(K_1)E(F_1)E(\mathbf Q_2)\subseteq E(K_2),
\]
which ultimately gives
\[
q(K_2)\leq 2.
\]

A second difference is the form of the final hypothesis. Kumakawa assumes the
class-group condition
\[
r_4(A_2)=1.
\]
This is already very close in spirit to our argument. The new feature of our result is that we replace
such a condition on $A_2$ by explicit numerical conditions. Namely, the
quartic-symbol assumptions
\[
\left(\frac{2}{p}\right)_4
\left(\frac{2}{q}\right)_4
\left(\frac{pq}{2}\right)_4=-1
\qquad\text{and}\qquad
\left(\frac{2}{p}\right)_4=-1
\ \text{ or }\
\left(\frac{2}{q}\right)_4=-1
\]
are used to prove the required bound on the Hasse unit index, and hence the
required smallness of $A_2$. 

It seems reasonable to expect that an analogous refinement should be possible
in Kumakawa's case. More precisely, one should be able to analyze the relevant
unit index and R\'edei-type data in the family
\[
p \equiv 3 \mod 8, \qquad q\equiv 9 \mod {16}, \qquad
\left(\frac{q}{p}\right)=-1, \qquad
\left(\frac{2}{q}\right)_4=1,
\]
and thereby obtain explicit quartic-symbol conditions which imply
$r_4(A_2)=1$. Such a criterion would turn Kumakawa's result into a fully
explicit numerical criterion.
\end{remark}

\subsection*{Acknowledgments} The author Josu\'e \'Avila would like to thank the Department of Mathematics at Universidad de Costa Rica, the "Vicerrector\'ia de Investigaci\'on" for funding his research under the project 821-C4-257, and the "Centro de Investigaci\'on de Matem\'atica Pura y Aplicada" (CIMPA) at Universidad de Costa Rica where he is currently employed as a research professor.

\printbibliography

\end{document}